\documentclass[12pt, leqno]{article}
\usepackage{amsmath, amssymb, amscd, verbatim, xspace,amsthm}
\usepackage{latexsym, epsfig, color}

\newcommand{\Z}{\ensuremath{{\mathbb{Z}}}\xspace}
\renewcommand{\P}{\ensuremath{{\mathbb{P}}}}
\newcommand{\Q}{\ensuremath{{\mathbb{Q}}}}
\newcommand{\R}{\ensuremath{{\mathbb{R}}}}

\newcommand{\ra}{\rightarrow}

\newcommand\Det{\operatorname{Det}}
\newcommand\Hom{\operatorname{Hom}}

\newcommand\sHom{\operatorname{\mathcal{H}om}}

\newcommand\im{\operatorname{im}}

\newcommand\Sym{\operatorname{Sym}}

\newcommand\tensor{\otimes}
\newcommand\isom{\cong}
\newcommand\sub{\subset}
\newcommand\tesnor{\otimes}
\newcommand\wn{\wedge^n}
\newcommand\disc{\operatorname{disc}}

\newcommand\GL{\operatorname{GL}}

\newcommand\Spec{\operatorname{Spec}}
\newcommand\Proj{\operatorname{Proj}}

\newcommand\wt{\wedge^2}

\renewcommand\O{\mathcal{O}}

\newcommand\BS{\ensuremath{S}\xspace}

\newcommand\fIf{\ensuremath{{\mathcal{I}_f}}\xspace}

\newcommand\OS{\ensuremath{{\O_\BS}}\xspace}

\newcommand\bij[2]{\ensuremath{\left\{\parbox{2.5 in}{#1}\right\} \longleftrightarrow \left\{\parbox{2.5 in}{#2}\right\}}}

\newcommand\bq{\begin{equation}}
\newcommand\eq{\end{equation}}

\newtheorem{proposition}{Proposition}[section]
\newtheorem{theorem}[proposition]{Theorem}
\newtheorem{corollary}[proposition]{Corollary}

\newtheorem{lemma}[proposition]{Lemma}

\theoremstyle{remark}

\renewcommand{\R}{\ensuremath{{\mathcal{R}}}}

\newcommand{\M}{\mathcal{M}}
\newcommand{\I}{\mathcal{I}}

\usepackage{fullpage}

\newenvironment{definition}{\vspace{2 ex}{\noindent{\bf Definition. }}}{\vspace{2 ex}}
\newenvironment{notation}{\vspace{2 ex}{\noindent{\bf Notation. }}}{\vspace{2 ex}}

\title{Parametrization of ideal classes in rings associated to binary forms}

\author{Melanie Matchett Wood\thanks{email:mwood@math.stanford.edu}\\
American Institute of Mathematics and Stanford University}

\begin{document}

\maketitle

\abstract{We give a parametrization of the ideal classes of rings associated to integral binary forms by
classes of tensors in $\Z^2\tensor \Z^n\tensor \Z^n$.  This generalizes Bhargava's work on Higher Composition Laws, which gives such parametrizations in the cases $n=2,3$.  
We also obtain parametrizations of 2-torsion ideal classes by symmetric tensors.  Further, we give versions of these theorems when $\Z$ is replaced by an arbitrary base scheme $S$, and geometric constructions of the modules from the tensors in the parametrization.}

\section{Introduction}\label{intro}

The goal of this paper is to give a parametrization of ideal classes in the rings associated to
binary $n$-ic forms.  
Every integral binary form of degree $n$ has a ring of rank $n$ (a ring isomorphic to $\Z^n$ as a $\Z$-module) associated to it (see, for example, \cite{Naka}, \cite{Simon}, \cite{Bf}).  When the form is irreducible over $\Q$, the associated ring is an order in a degree $n$ number field.
For $n=2$, the  ring associated to a binary quadratic form is just the quadratic ring of the same discriminant.
For $n=3$, binary cubic forms parametrize cubic rings exactly (\cite{DF}, \cite{GGS}).  
For $n>3$ not all rank $n$ rings (or orders in degree $n$ number fields) are associated to a binary $n$-ic form.  
The special orders which are associated to binary $n$-ic forms have been studied as a natural generalization of monogenic orders in \cite{Simon}, and in \cite{PrimeSplit} 
it is found that prime splitting in these orders can be understood simply in terms of the factorization of the form modulo the prime.
In  \cite{Bf} it is shown that binary $n$-ic forms parametrize rank $n$ rings that have an ideal class with certain special structure.
When $n=4$, there is another particularly simple description given in 
\cite{BinQuartic}; that is, binary quartic forms parametrize quartic rings with monogenic cubic resolvents.  The cubic resolvent
of of a quartic ring is an integral model of the classical cubic resolvent field, and was introduced in \cite{HCL3} to parametrize quartic rings.
Thus the orders associated to binary forms are interesting examples of orders in number fields with various nice properties.  In this paper, we parametrize the ideal classes of these orders by classes of tensors.

There are very few classes of orders whose ideal classes have been parametrized.  
When $n=2$, ideal classes of quadratic rings are parametrized by binary quadratic forms, originally by 
Dedekind and Dirichlet \cite{DD} (see also \cite{HCL1} for another parametrization of ideal classes in this case).  Bhargava \cite{HCL2} has found a parametrization of ideal classes of cubic rings.  This is his space of
2 by 3 by 3 boxes of integers, or classes of elements of $\Z^2\tensor \Z^3\tensor \Z^3$.  
The classical parametrization of ideal classes of quadratic rings by binary quadratic forms has had very important applications in number theory,
from genus theory to the computation of class groups of quadratic fields.  Bhargava's parametrization of ideal classes of cubic rings
is much more recent, but already it can be applied to find the average size of the 2-torsion part of the class group of cubic orders, 
both for maximal and non-maximal orders (see forthcoming work of Bhargava).
In this paper, we prove that classes of elements of 
$\Z^2\tensor \Z^n\tensor \Z^n$ parametrize ideal classes of the rings associated to binary $n$-ic forms for all $n$.
This gives an explicit parametrization of class groups of a natural infinite family of orders in rank $n$ number fields.
When $n=2,3$, these are the results of Bhargava in \cite{HCL1}, \cite{HCL2}.  

One can also study symmetric elements of
$\Z^2\tensor \Z^n\tensor \Z^n$, that is elements of $\Z^2\tensor \Sym_2 \Z^n$.  We relate these symmetric tensors to the 2-part of the class group
of rings associated to binary $n$-ic forms, just as in the cases $n=2,3$ in \cite{HCL1}, \cite{HCL2}.  Morales 
(\cite{Morales}, \cite{Morales2}) has also studied elements of $\Z^2\tensor \Sym_2 \Z^n$ and associated modules to them that are related to the $2$-part of certain class groups, though
he associates modules for a slightly different ring than that in our work.  Hardy and Williams
\cite{HW} have also given class number formulas counting elements of $\Z^2\tensor \Sym_2 \Z^2$ of a given discriminant.

In addition, this paper gives analogous results when the integers are replaced by an arbitrary base scheme $S$ 
(or base ring when $S=\Spec R$), and so we also generalize the results from \cite{HCL1} and \cite{HCL2} from the integers to an arbitrary base.  In Morales's work \cite{Morales}, he has replaced $\Z$ by an arbitrary maximal order in a number field in his constructions of modules
from symmetric tensors.  

In this paper, we give both algebraic and geometric constructions for the modules associated to an element of  $\Z^2\tensor \Z^n\tensor \Z^n$.
The algebraic construction is given by explicit formulas for the action of the ring elements on a $\Z$-basis for an ideal.  The geometric construction
gives the modules as sections of line bundles of schemes naturally given by the tensor, and is quite simple for nice tensors.  

As this paper generalizes the results of \cite{HCL2} from $\Z$ to an arbitrary base, it could be applied
to counting problems in number theory over arbitrary orders in number fields or function fields, specifically to finding the average of the
the 2-part of the class group of orders that are cubic over a fixed order.
For $n> 3$, the associated counting problems for elements of $\Z^2\tensor \Z^n\tensor \Z^n$ are much harder, and involve reduction theory problems at the boundary of current research.
Given how little is known about the average size of 2-parts of class groups in any natural infinite family of degree $n$ number fields, the application
of this paper and new work for the reduction theory  of $\Z^2\tensor \Z^n\tensor \Z^n$ has 
 potentially very interesting results.  

This work can also be interpreted geometrically, because a rank $n$ algebra over a scheme $S$ is equivalent to a finite, flat degree $n$ cover of $S$.  In those terms, this work explicitly constructs the moduli space of line bundles (and degenerations) on certain finite covers.  In the case $n=3$, this includes all cubic covers.  In forthcoming work of the author and Erman, this paper is applied over $S=\P^1$ to study explicitly the universal Jacobian of the Hurwitz stack of trigonal curves.  The explicit description of the moduli space provided by this paper gives a proof of unirationality of the space as well as a framework for computing its Picard group  and possibly its Chow ring.

\subsection{Outline of results}

We can represent an element of the space $\Z^2\tensor \Z^n\tensor \Z^n$ as a pair $A=(A_1,A_2)$ of $n$ by $n$
matrices.  Let $\Det(A)$ be the binary $n$-ic form $\Det(A_1x+A_2y)$.  For a non-zero integral binary form $f$, let $R_f$ be the ring
of global functions of the subscheme of $\P^1_\Z$ cut out by $f$.  (The same ring, via other constructions, is associated to $f$ in 
\cite{PrimeSplit}, \cite{Naka}, \cite{Simon}, and \cite{Bf} and reviewed in Section~\ref{S:binf} of the current paper.)  
A binary form $f$ is \emph{non-degenerate} if it has non-zero discriminant.
Let $G$ be the subgroup of elements $(g_1,g_2)$ of $\GL_n(\Z) \times \GL_n(\Z)$ such that $\Det(g_1)\Det(g_2)=1$.
Over the integers, we have the following version of our main theorem which makes some assumptions on $f$ for simplicity of statement.  (This version follows from Theorem~\ref{T:bijZ} using 
Propositions~\ref{P:eqb}, \ref{P:charfrac}, and \ref{P:primu}.)

\begin{theorem}\label{T:specbijI}
For a primitive non-degenerate binary $n$-ic form $f$, there is a bijection between $G$-classes of $A\in \Z^2\tensor \Z^n\tensor \Z^n$ such that $\Det(A)=f$ and \emph{(}not necessarily invertible\emph{)} ideal classes of $R_f$. 
\end{theorem}

If $f$ is monic and $\Det(A)=f$, then $\Det(A_1)=1$.  We can then act by an element of $G$ so as to assume that
$A_1$ is the identity matrix.  Further $G$ action fixing $A_1$ (the identity matrix) is just conjugation of $A_2$.
If $f=F(x_1,x_2)$ is monic, then $R_f=\Z[\theta]/F(\theta,1)$, 
and Theorem~\ref{T:specbijI} generalizes the classical result
that ideal classes of monogenic orders correspond to conjugacy classes in $\Z^n\tensor \Z^n$ whose characteristic polynomial is
$F(t,-1)$.
So, we can view Theorem~\ref{T:specbijI} as placing rings associated to binary forms in analogy with monogenic rings,
as in \cite{PrimeSplit} and \cite{Simon}.

In the case $n=3$, Theorem~\ref{T:specbijI} is slightly stronger than the corresponding version in \cite{HCL2},
which gives a correspondence between $A$ associated to invertible ideals and invertible ideal classes of $R_f$.

As in \cite{HCL1,HCL2}, we must define a notion of balanced to state a more general theorem.
We will show in Section~\ref{S:balance} that there are several equivalent ways to formulate the notion of balanced, but we first give the one closest to Bhargava's notion in \cite{HCL1,HCL2}.  
For a non-zero form $f$, there is a naturally associated ideal $I_f$ of $R_f$.
This ideal class was constructed in \cite{Bf} and the construction is reviewed in Section~\ref{S:binf} of the current paper.
A \emph{balanced} pair of ideals (for $f$) is a pair $(M,N)$ of fractional $R_f$-ideals such that 
$MN\subset I_f$ and $|M||N|=|I_f|$, where $|P|$ denotes the norm of $P$.  Two such pairs $(M,N)$ and $(M',N')$ are in the same class if
$M'=\lambda M$ and $N'=\lambda^{-1} N$ for some invertible element 
$\lambda\in R_f\tensor_\Z \Q$.  Note that an invertible ideal has a unique balancing partner, and thus classes of balanced pairs of invertible ideals are the same as invertible ideal classes.

\begin{theorem}\label{T:IsameBhar}
For non-degenerate $f$, we have a bijection
$$
\bij{classes of balanced pairs $(M,N)$ for $f$}{$G$-classes of $A\in \Z^2 \tensor \Z^n \tensor \Z^n$ with $\Det(A)=f$}.
$$
\end{theorem}

Theorem~\ref{T:IsameBhar} follows from Theorems~\ref{T:bijZ} and \ref{T:sameBhar}.
 In \cite{HCL2}, Bhargava asks for an appropriate formula of balanced
for degenerate forms so as to obtain a theorem such as the above.  We will give such a notion, which can be seen for example in Theorem~\ref{T:introarb}.

We  also relate symmetric elements of $A\in \Z^2 \tensor \Z^n \tensor \Z^n$ to the $2$-part of class groups of rings
associated to binary forms.  For nice forms $f$ we get the following, which follows from Theorems~\ref{T:gensym} and \ref{T:primequiv}.
\begin{theorem}\label{T:primsym}
For every primitive non-degenerate binary $n$-ic form $f$ with coefficients in $\Z$, there is a bijection
$$
\bij{classes of $(M,k)$ where $M$ is a fractional $R_f$-ideal,
$k$ is an invertible element of $R_f\tensor_\Z \Q$, and $M=(I_fk:M)$}{$\GL_n(\Z)$ classes of $A\in \Z^2 \tensor \Sym_2 \Z^n $ with $\Det(A)=f$},
$$
where $\Sym_2 \Z^n$ are symmetric $n$ by $n$ matrices, the action of $g\in\GL_n(\Z)$ is by multiplication on the left by $g$ and right by $g^t$,
and $(M,k)$ and $(M_1,k_1)$ are in the same class if
$M_1=\lambda M$ and $k_1=\lambda^{2} k$ for some invertible element 
$\lambda\in R_f\tensor_\Z \Q$, and $(I_fk:M)$ is the fractional ideal of elements $x$ such that
$xM\subset I_fk$.
\end{theorem}
We also give a version of the above in Theorem~\ref{T:gensym} for all non-zero forms $f$, which again uses a more delicate notion of balanced.  
If we restrict Theorem~\ref{T:primsym} to invertible modules $M$, then the condition $M=(I_fk:M)$ is replaced by
$M^2=I_fk$, and the restricted set is an extension of a torsor of the 2-part of the class group of $R_f$ by $R_f^*/R_f^2$.
(We say a torsor instead of a principal homogeneous space because $I_f$ might not be a square in the class group and there
would be no such $M$ in that case.)

This paper also gives analogous results over an arbitrary base scheme $S$.  We consider $V,U,W$, locally free $\OS$-modules of ranks 2, $n$, and $n$,
respectively.  We then study global sections $p\in V\tesnor U \tesnor W$.  We can construct 
$\Det(p)\in \Sym^n V \tensor \wn U \tensor \wn W$, which is a binary $n$-ic form.  
Fix any $f$ in $\Sym^n V \tensor L$, where $L$ is a locally free rank 1 $\OS$-module.
There is a natural associated rank $n$ $\OS$-algebra $R_f$ and an $R_f$ module $I_f$ with a natural map $I_f\ra V$ of $\OS$-modules (see \cite{Bf}, or Section~\ref{S:arbbase} of the current paper).
 A \emph{balanced pair of modules} for a non-zero-divisor $f$ is a 
pair of $R_f$-modules $M$ and $N$, each a locally free rank $n$ $\OS$-module such that $\wn M \tensor \wn N\isom L^*$, and a map of $R_f$-modules $M \tensor_{R_f} N \ra I_f$, such that when the composition
$M \tensor_\OS N \ra M \tensor_{R_f} N \ra I_f \ra V$ is written as $A\in M^* \tesnor N^* \tensor V$
 we have $\Det(A)=fu$, where $u$ is a unit in $\OS$.  (We can then choose a unique isomorphism $\wn M \tensor \wn N\isom L^*$ such that $\Det(A)=f$.)
We have the following, proven in Theorem~\ref{T:ARB}.
\begin{theorem}\label{T:introarb}
For every non-zero-divisor binary $n$-ic form $f\in \Sym^n V \tensor L$, there is a bijection
$$
\bij{isomorphism classes of balanced pairs $(M,N)$ of modules for $f$}{isomorphism classes of $A\in V\tensor U\tensor W$, where $U$ and $W$ are locally free rank $n$ $\OS$-modules with
an isomorphism $\wn U \tensor \wn W\isom L$ such that $\Det(A)=f$}.
$$
\end{theorem}

From a $p\in V\tesnor U \tesnor W$ we give two constructions of the corresponding ideal
classes or modules.  The first construction (in Section~\ref{S:arbbase}) is algebraic and explicit and the second (in Section~\ref{S:ArbBase}) is geometric and more intuitive.  
We give a heuristic description of the geometric construction here.  If we have locally free $\OS$-modules $F$ and $G$,
and $s\in F \tesnor G$, then we can construct the $k$-minor 
$\wedge^k s \in \wedge^k F \tensor \wedge^k G$.  If $H$ is also a locally free $\OS$-module, and
we have $s\in F \tesnor G\tensor H$, then we have a $k$-minor $\wedge^k_H s$ with $H$-coefficients in $\wedge^k F \tensor \wedge^k G \tensor \Sym^k H$ (see \cite[Section 8.2]{Quartic} for the details of this construction).
For $p\in V\tesnor U \tesnor W$, the $n$-minor with coefficients in $V$ defines a subscheme
$T_p(V)$ in $\P(V)$, the $2$-minor with coefficients in $U$ defines a subscheme
$T_p(U)$ in $\P(U)$, and the $2$-minor with coefficients in $W$ defines a subscheme
$T_p(W)$ in $\P(W)$.  Abusing notation, we let $\pi$ denote the map from all of these schemes to $S$.
The \emph{heuristic} definition of $R_f$ is to take $\pi_* \O_{T_p(V)}$ (or $\pi_* \O_{T_p(U)}$ or $\pi_* \O_{T_p(W)}$--all the $\OS$-algebras turn out to be the same in the nicest cases),
and $M=\pi_* \O_{T_p(U)}(1)$ and $N=\pi_* \O_{T_p(W)}(1)$ (where $\O(1)$ is as pulled back from the corresponding projective bundle).
This construction does not work for all $p$ (e.g. the three algebras given for $R_f$ are not necessarily the same)  and it is not functorial in $S$.  As in the case of binary $n$-ic forms in \cite{Bf}, we use hypercohomology to extend our heuristic geometric construction to a construction that works in all cases and is functorial.

\subsection{Outline of the paper}
In Section \ref{S:binf} we review the rings and ideals associated to binary $n$-ic forms, and give some computations with these rings and ideals that are critical for the work in this paper.  In Section~\ref{S:Main}, we prove our main theorem (Theorem~\ref{T:bijZ}) over the integers.  We first give the algebraic constructions of a pair of modules from an element of $\Z^2 \tesnor \Z^n \tensor \Z^n$ in Section~\ref{SS:algZ}.
In Section~\ref{SS:proofZ} we prove Theorem~\ref{T:bijZ} when
the leading coefficient of $f$ is not zero.
In Section~\ref{S:GLinv}, we study the $\GL_2(\Z)$ 
invariance of our construction of modules, and use this to finish the proof of Theorem~\ref{T:bijZ}.
In Section~\ref{S:sym} we give the general analogs of Theorem~\ref{T:bijZ} for symmetric tensors.

In Section~\ref{S:balance}, we further study the notion of balanced 
pairs of modules, and show it is equivalent to a characteristic 
polynomial condition and an index condition (Proposition~\ref{P:eqb}).
In Section~\ref{SS:nodeg}, we show that for non-degenerate binary $n$-ic forms 
all balanceable modules are fractional ideals (Proposition~\ref
{P:charfrac}), and prove that in this case the definition of balanced modules
is equivalent to the definition of balanced ideals given above (Theorem~\ref
{T:sameBhar}).    In Section~\ref{SS:prim}, we specialize to 
the case of primitive non-degenerate forms, where we see that every 
fractional ideal has a unique balancing partner.  This is the final 
step in the proof of Theorem~\ref{T:specbijI}. 

In Section~\ref{S:arbbase}, we prove versions of these main theorems over an arbitrary base.  In particular, we prove Theorem~\ref{T:introarb} (as Theorem~\ref{T:ARB}) and a symmetric version.
In Section~\ref{SS:geomZ}, we give
a geometric construction of the modules 
from the universal tensor
and prove it is the same as the algebraic construction in Section~\ref{SS:algZ}.  The main obstacle is that we give multiple
ring constructions and we must show that they agree.  The rings
are given by global sections of different schemes, but the schemes themselves are not isomorphic.
Finally, in Section~\ref{S:ArbBase}, we give a geometric construction over arbitrary base of the modules from a triple tensor
and prove that it commutes with base change (Corollary~\ref{C:BC}).

\section{Binary forms, rings, and ideals}\label{S:binf}
Given a binary $n$-ic form $f_0x_1^n +f_1 x_1^{n-1}x_2 +\dots +f_nx_2^n$ with $f_i\in \Z$, there is a naturally associated
rank $n$ ring $R_f$ (see \cite{Naka}, \cite{Simon}, \cite{Bf}) and a sequence of $R_f$-modules (see \cite{SimonIdeal}, \cite{Bf}). 
Here we review the facts from \cite{Bf} about $R_f$ and these
modules that are necessary in this paper, as well as make some computations that will be critical in this work.  We will eventually need these results over more general rings than $\Z$,
so we will now work over an arbitrary ring $B$ in place of $\Z$.  

Let $f=f_0x_1^n +f_1 x_1^{n-1}x_2 +\dots +f_nx_2^n$ a binary $n$-ic form with coefficients $f_i\in B$ such that $f$
is not a zero divisor in $B[x,y]$.  We first give geometric constructions of a ring and ideals from $f$, as given in \cite{Bf}.
We define $R_f$ as the $B$-algebra of global sections of the regular functions of $T_f$,
the subscheme of $\P^1_B$ defined by $f$.  We have line bundles $\O_{T_f}(k)$ on $T_f$ pulled back from $\O(k)$ on $\P^1_B$.
We define $I_f$ to be $\Gamma(\O_{T_f}(n-3))$ (i.e. the global sections of $\O_{T_f}(n-3)$), and 
$J_f$ to be $\Gamma(\O_{T_f}(n-2))$.  This gives $I_f$ and $J_f$ the structure of $R_f$-modules. 
Note that our $I_f$ is the $I_f^{n-3}$ or $\fIf_{n-3}$ of \cite{Bf}, and our $J_f$ is the $I_f^{n-2}$ or $\fIf_{n-2}$ of \cite{Bf}.

Equivalent, but more concrete, constructions of $R_f$, $I_f$, and $J_f$ are also given in \cite{Bf}, and we give those now,
as they will be easier to work with. 
For these constructions, we assume $f_0$ is not a zero-divisor in $B$.
 Write $f=F(x,y)$.  
Let $B'=B_{f_0}$ (the ring $B$ with $f_0$ inverted).  We can also define the $B$-algebra 
$R_f$ as the subring of $B'[\theta]/F(\theta,1)$ generated by $\zeta_0,\dots,\zeta_{n-1}$ with 
$$\zeta_0=1, \text{   and   }
\zeta_k=f_0\theta^k+\dots+f_{k-1}\theta\text{ for } k>0.$$
The $\zeta_k$ give a $B$-module basis of $R_f$, and it is shown in \cite[Theorem 2.4]{Bf} that 
this definition of
$R_f$
agrees with the geometric one give above, and in particular that the $B$-module generated by the $\zeta_i$
 is closed under multiplication.
Note that if $f_0$ is a unit in $B$, then $R_f=B[\theta]/F(\theta,1)$.
We can define $I_f$ and $J_f$ as sub-$B$-modules of $B'[\theta]/F(\theta,1)$, such that 
\begin{align*}
I_f 
  \mbox{ is the }  &B \mbox{-module generated by } && 1,\theta,\theta^2,\dots,\theta^{n-3},\zeta_{n-2},\zeta_{n-1} \textrm{ or}\\
\mbox{ equiv. the }&B\mbox{-module generated by } && 1,\theta,\theta^2,\dots,\theta^{n-3},f_0\theta^{n-2},f_0\theta^{n-1}+f_1\theta^{n-2} \\
J_f 
  \mbox{ is the }&B\mbox{-module generated by } && 1,\theta,\theta^2,\dots,\theta^{n-3},\theta^{n-2},\zeta_{n-1}. &
\end{align*}
When $n=2$, we use only the second description of $I_f$ given above.
In \cite[Theorem 2.4]{Bf}, it is shown that these definitions of $I_f$ and $J_f$ agree with the geometric ones given above, and in particular that
$I_f$ and $J_f$ are closed under multiplication by elements of $R_f$.
We have a map of $R_f$-modules $I_f\ra J_f$ given by inclusion.  This map is not canonical and does not arise geometrically,
yet it will be important in our proofs.  

The elements $f_0,\zeta_1,\dots,\zeta_{n-1}$ are a 
$B'$-module basis of 
$B'[\theta]/F(\theta,1)$.
Let $\check{\zeta_i}$ be the $B'$-module basis of $\Hom_{B'}(B'[\theta]/F(\theta,1),B')$ 
dual to the $\zeta_i$.
So $\check{\zeta}_i(\zeta_j)=\delta_{ij}$ for $j>0$.  
Also, let $\check{\theta}_i$ be the $B'$-module basis of $\Hom_{B'}(B'[\theta]/F(\theta,1),B')$ 
dual to $1,\theta,\theta^2,\dots,\theta^{n-1}$.
We can apply
these $\check{\zeta}_i$ and $\check{\theta}_i$ to elements in $I_f$ and $J_f$ since they lie in $B'[\theta]/F(\theta,1)$,
but they are not necessarily dual to a $B$-module basis of $I_f$ or $J_f$.
The following are the key computations we will need.

\begin{proposition}\label{P:getcoeff}
For $r\in B'[\theta]/F(\theta,1)$ and $1\leq k\leq n-1$, 
$$\check{\zeta}_{n-1}(\zeta_k r)=\check{\theta}_{n-1-k}(r)-f_k\check{\zeta}_{n-1}(r).$$
\end{proposition}
\begin{proof}
We will write out $\zeta_k r$ in terms of powers of $\theta$ and then read off the
coefficient of $\theta^{n-1}$.  First, we write $r=\sum_{j=0}^{n-1} r_j\theta^j$ and so
$$
\zeta_k r=(f_0\theta^k+\dots+f_{k-1}\theta)(r_{n-1}\theta^{n-1}+\dots+r_0).
$$
To find the $\theta^{n-1}$ coefficient, we only have to look at terms of $r$ with 
$j\geq n-1-k$.  From the $r_{n-1-k}\theta^{n-1-k}$ term we get a $\theta^{n-1}$ coefficient
of $r_{n-1-k}f_0$.  From the remaining terms, we get the sum
\begin{align*}
\sum_{j=n-k}^{n-1} r_j\theta^j(f_0\theta^k+\dots+f_{k-1}\theta) 
&=\sum_{j=n-k}^{n-1} r_j\theta^{j-(n-k)}(f_0\theta^{n}+\dots+f_{k-1}\theta^{n-k+1})\\
&=\sum_{j=n-k}^{n-1} -r_j\theta^{j-(n-k)}(f_k\theta^{n-k}+\dots+f_{n})\\
\end{align*}
and the only term of the final sum with a non-zero $\theta^{n-1}$ coefficient
is the $j=n-1$ term which has a $\theta^{n-1}$ coefficient of $-r_{n-1}f_k$.
So $\check{\theta}_{n-1}(r)=f_0\check{\theta}_{n-1-k}(r)-f_k\check{\theta}_{n-1}(r)$,
and dividing by $f_0$ proves the proposition.
\end{proof}

\begin{corollary}\label{C:thet}
For $r\in B'[\theta]/f(\theta,1)$, $$\check{\zeta}_{n-1}(\theta r)=\check{\zeta}_{n-2}(r).$$
\end{corollary}
\begin{proof}
We have
  \begin{equation*}
  \check{\zeta}_{n-1}:=\frac{\check{\theta}_{n-1}}{f_0} \quad \mbox{and} \quad 
  \check{\zeta}_{n-2}:=\frac{\check{\theta}_{n-2}- \frac{f_1 \check{\theta}_{n-1}}{f_0} }{f_0},
  \end{equation*}
and thus this follows from the above proposition when $k=1$.
\end{proof}

\begin{lemma}\label{L:inI}
If we have a homomorphism $\phi$ of $B$-modules from some $B$-module $P$ to $J_f$,
then the image of $\phi$ is in $I_f$ if and only if the image of 
$\check{\zeta}_{n-2}\phi$ is in $B$.
\end{lemma}
\begin{proof}
The elements of
$I_f$ are just the elements $j\in J_f$ for which $\check{\zeta}_{n-2}(j)\in B$.
\end{proof}

Thus $\check{\zeta}_{n-1}$ and $-\check{\zeta}_{n-2}$ give two $B$-module maps
from $I_f$ to $B$, or a $B$-module map $I_f \ra V=B^2$, where $(-1)^{i+1}\zeta_{n-i}$ gives the map into
the $i$th coordinate of $V$.  We have chosen the maps in such a way that the map $I_f \ra V$ the canonical
map given in \cite[Equation (3.8)]{Bf} (where the $k$ in \cite[Equation (3.8)]{Bf} is $n-3$ for our purposes).  
This map is useful because it doesn't lose information about $R_f$-module maps.  More formally,
we have the following.

\begin{proposition}\label{P:Istruc}
 For any binary $n$-ic form $f$ and any $R_f$-module $P$,
composition with the map $I_f \ra V$ gives an injection
of $R_f$ modules
$$
\Hom_{R_f}(P,I_f)\ra \Hom_B(P,V).
$$
\end{proposition}

\begin{proof}
Suppose, for the sake of contradiction, that we had a non-zero map $\phi \in \Hom_{R_f}(P,I_f)$ such that
the image of $\phi$ was in the kernel of $I_f\ra V$.  Let $r$ be a non-zero element of $\im(\phi)$.  
Then, by Proposition~\ref{P:getcoeff} we have 
$$\check{\theta}_{n-1-k}(r)=\check{\zeta}_{n-1}(\zeta_k r)=0$$
for $2\leq k\leq n-1$.  Thus, we see that $r=0$.
\end{proof}

In \cite[Corollary 3.7]{Bf}, it is shown that as $R_f$ modules, $J_f \isom \Hom_B(R_f,B)$, and thus we have the following proposition.
\begin{proposition}\label{P:Jstruc}
For any binary $n$-ic form $f$ and any $R_f$-module $P$,
composition with the map $\check{\zeta}_{n-1} \colon J_f \ra B$ gives an isomorphism
of $R_f$ modules
$$
\Hom_{R_f}(P,J_f)\stackrel{\check{\zeta}_{n-1}}{\ra} \Hom_B(P,B).
$$
\end{proposition}

\section{Main theorems}\label{S:Main}

We write an element $A\in \Z^2 \tensor \Z^n \tensor \Z^n$ as pair $A_1,A_2$ of $n\times n$ matrices.
The \emph{determinant} of $A$ is the binary $n$-ic form $\Det(A_1x_1 +A_2x_2)$.
For non-zero binary form $f$ with integer coefficients, we defined in Section~\ref{S:binf}
 a rank $n$ ring $R_f$ and two modules $I_f$ and $J_f$ for that ring.
Recall that we have a map $I_f \ra V$ of abelian groups, where $V=\Z^2$.
We will next define a notion of a balanced pair of $R_f$-modules.  The idea is that the product of the
pair should map to $I_f$, but that map should be constrained by the form $f$ itself.
Note the definition we now give is different from the one given in the Introduction, but we will see in Theorem~\ref{T:sameBhar} that the definitions agree for non-degenerate $f$.

\begin{definition}
 A \emph{based balanced pair of modules} for $f$ is a 
pair of $R_f$-modules $M$ and $N$, a choice of basis
$M \isom \Z^n$ and $N \isom \Z^n$, and a map of $R_f$-modules
$M \tensor_{R_f} N \ra I_f$, such that when the composition
$M \tensor_\Z N \ra M \tensor_{R_f} N \ra I_f \ra V$ is written as a pair of matrices
$A_1$ and $A_2$ (viewing elements of $M$ as row vectors and elements of $N$ as column vectors), we have $\Det(A_1 x_1 + A_2 x_2)=f$.
If $v_i$, $m_j$, and $n_k$ are the bases of $V$, $M$, and $N$ respectively
indicated above, then
the $j,k$ entry of $A_i$ is the coefficient of $v_i$ in the image of $m_j \tesnor n_k$, i.e.
$(-1)^{i+1}\zeta_{n-i}(m_j \tesnor n_k)$.
We will often refer to the based balanced pair as $M, N$, with the bases and balancing map understood.
\end{definition}

\begin{definition}
A \emph{balanced pair of modules} for a non-zero form $f$ is a 
pair of $R_f$-modules $M$ and $N$, each a free rank $n$ $\Z$-module,
and a map of $R_f$-modules $M \tensor_{R_f} N \ra I_f$, such that when the composition
$M \tensor_\Z N \ra M \tensor_{R_f} N \ra I_f \ra V$ is written as a pair of matrices $A_1$ and $A_2$,
 we have $\Det(A_1 x_1 + A_2 x_2)=\pm f$.
Given a balanced pair of modules for a non-zero form $f$, there is
a unique choice of generator $\chi$ of $\wn M \tensor \wn N$
such that $\Det(A_1 x_1 + A_2 x_2)=f$
when constructing $A$ with bases of $M$ and $N$ that give $\chi\in \wn M \tensor \wn N$. 
If we have based balanced pairs $(M,N)$ and $(M',N')$ such that
the modules and balancing maps are the same and only the bases differ, then the change of bases must preserve $\chi$ since
both based balanced pairs give $\Det(A_1 x_1 + A_2 x_2)=f$.
\end{definition}

In this section we prove the following theorem.
\begin{theorem}\label{T:bijZ}
For every non-zero binary $n$-ic form $f$ with coefficients in $\Z$, there is a bijection
$$
\bij{based balanced pairs $(M,N)$ of modules for $f$}{$A\in \Z^2 \tensor \Z^n \tensor \Z^n$ with $\Det(A)=f$}.
$$
Let $G$ be the subgroup of $\GL_n(\Z)\times\GL_n(\Z)$ of elements $(g_1,g_2)$ such that
$\Det(g_1)\Det(g_2)=1$.  Then, $G$ acts equivariantly in the above bijection
(acting of the bases of $M$ and $N$), and we obtain a bijection
$$
\bij{isomorphism classes of balanced pairs $(M,N)$ of modules for $f$}{$G$-classes of $A\in \Z^2 \tensor \Z^n \tensor \Z^n$ with $\Det(A)=f$}.
$$
\end{theorem}
We now give  a map $\phi$ from balanced based pairs to $\Z^2 \tensor \Z^n \tensor \Z^n$.  From the definition of a balanced based pair, we have
the map of $\Z$-modules
$M \tensor_\Z N \ra M \tensor_{R_f} N \ra I_f \ra V$, which can be written as a pair of matrices
$A_1,A_2$ as above.  This pair of matrices is $\phi(M,N)$.

In Section~\ref{SS:algZ} we construct a based balanced pair of modules from an $A\in \Z^2 \tensor \Z^n \tensor \Z^n$.
Our construction is completely concrete, and we give formulas for the action of $R_f$ on $M$ and $N$.
 In Section~\ref{SS:proofZ}, we prove that this construction gives an inverse to the map $\phi$ described above when $f_0\ne0$.
 In Section~\ref{S:GLinv}, we use the $\GL_2$ equivariance of our construction to reduce to the case that $f_0\ne0$, which will prove
Theorem~\ref{T:bijZ}.

\subsection{Construction of balanced pair of modules}\label{SS:algZ}

We are given $A\in \Z^2 \tensor \Z^n \tensor \Z^n$, which we can write as a pair $A_1, A_2$ of $n\times n$ matrices.
Let $f$ be the determinant of $A$.  In this section, we will construct a based balanced pair $(M,N)$ of modules for $f$.
We begin by letting $M=\Z^n$ and $N=\Z^n$ as abelian groups.  It remains to specify
the $R_f$ action on $M$ and $N$ and the map of $R_f$-modules $M \tensor_{R_f} N \ra I_f$.
We can write the elements of $M$ as row vectors with entries in $\Z$ and the elements of $N$ as column vectors with entries in $\Z$.
Heuristically, the action of $R_f$ will by given by $\theta$ acting on $M$ on the right by $-A_2A_1^{-1}$
and on $N$ on the left by $-A_1^{-1}A_2$.  The trouble with this construction is that $\theta$ is not an element of $R_f$
(unless $f_0$, the $x^n$ coefficient of $f$, is $\pm 1$) and that $A_1$ is not necessarily invertible 
(it could be the zero matrix!).  
We could solve both of these problems by inverting $f_0$, but it is possible that $f_0=0$.  
So we will pass to a universal situation, where we can always invert $f_0$.

We replace $\Z$ by the ring $\Lambda=\Z[\{u_{ijk}\}_{1\leq i\leq 2,1\leq j\leq n,1\leq k\leq n}]$ (the
free polynomial algebra on $2n^2$ variables over $\Z$), and we replace $A$ with the universal tensor $\mathcal{C}$ in $\Lambda^2 \tensor_\Lambda \Lambda^n \tensor_\Lambda \Lambda^n$,
where  $\mathcal{C}_i$ has $j,k$ entry $u_{i,j,k}$. 
We have a binary $n$-ic form $c=\Det(\mathcal{C}_1 x_1 + \mathcal{C}_2 x_2)$ with coefficients in $\Lambda$.  
We now let $M_\mathcal{C}=\Lambda^n$ and $N_\mathcal{C}=\Lambda^n$ as $\Lambda$-modules.  
We will give an action of the
$\Lambda$-algebra $R_c$ on $M_\mathcal{C}$ and $N_\mathcal{C}$
 and then we will give a map of $R_c$-modules $M_\mathcal{C} \tensor_{R_c} N_\mathcal{C} \ra I_c$.
This construction will be equivariant for the $G_\Lambda$ actions, where $G_\Lambda$ is the subgroup of $\GL_n(\Lambda)\times\GL_n(\Lambda)$
of $(g_1,g_2)$
such that
$\Det(g_1)\Det(g_2)=1$.
To recover a construction over $\Z$, we can just specialize by letting the $u_{i,j,k}=a_{i,j,k}$
in our formulas.

\subsubsection{$R_c$ action}\label{S:Raction}
We will write elements of $M_\mathcal{C}$ as row vectors with entries in $\Lambda$ 
and elements of 
$N_\mathcal{C}$ as column vectors with entries in $\Lambda$.
We can write $c=c_0 x_1^n + c_1 x_1^{n-1} x_2 + \dots + c_n x_2^n$.  
We will invert $c_0$ and denote all of the corresponding objects with a $'$.  For example, we have
$\Lambda'=\Lambda_{c_0}$, the ring $\Lambda$ with $c_0$ inverted.  We also have $R_c'=R_c\tensor_\Lambda \Lambda'$, which is just
the result of inverting $c_0$ in $R_c$.  If we write $c=C(x_1,x_2)$, we know from Section~\ref{S:binf} that $R_c'=\Lambda'[\theta]/C(\theta,1)$.
We have $M_\mathcal{C}'=M_\mathcal{C}\otimes_\Lambda \Lambda'$ and $N_\mathcal{C}'=N_\mathcal{C}\otimes_\Lambda \Lambda'$.

We define an action of $R_c'$ on $M_\mathcal{C}'$ and $N_\mathcal{C}'$ 
(which we still view as row vectors and column vectors respectively, just now with entries in $\Lambda'$)
by having $\theta$ act like 
$-\mathcal{C}_2\mathcal{C}_1^{-1}$ (on the right) on the row vectors and $-\mathcal{C}_1^{-1}\mathcal{C}_2$ (on the left) on the column vectors.
Since $\Det(\mathcal{C}_1x_1+\mathcal{C}_2x_2)=c$, the matrices $-\mathcal{C}_2\mathcal{C}_1^{-1}$ and  $-\mathcal{C}_1^{-1}\mathcal{C}_2$ satisfy their (common) characteristic polynomial $C(t,1)$.
Thus we have given a well-defined action
of $\Lambda'[\theta]/u(\theta,1)$ on $M_\mathcal{C}'$ and $N_\mathcal{C}'$.  
This restricts to an action of $R_c$ on $M_\mathcal{C}'$ and $N_\mathcal{C}'$, which we will now show
is actually an action of $R_c$ on $M_\mathcal{C}$ and $N_\mathcal{C}$.  

\begin{lemma}\label{L:nodenom}
  For $1\leq k \leq n-1$, the matrix
\begin{equation}
c_0(-\mathcal{C}_1^{-1}\mathcal{C}_2)^k + c_1(-\mathcal{C}_1^{-1}\mathcal{C}_2)^{k-1} + \dots c_{k-1}(-\mathcal{C}_1^{-1}\mathcal{C}_2)
\end{equation}
whose entries a priori are in  $\Q(u_{i,j,k})$ (the fraction field of $\Lambda$) are actually in $\Lambda$.
\end{lemma}
\begin{proof}
Over the field $\Q(u_{i,j,k})$, since $-\mathcal{C}_1^{-1}\mathcal{C}_2$ satisfies its characteristic polynomial, we have 
\begin{eqnarray*}
c_0(-\mathcal{C}_1^{-1}\mathcal{C}_2)^n &+& c_1(-\mathcal{C}_1^{-1}\mathcal{C}_2)^{n-1} + \dots c_{k-1}(-\mathcal{C}_1^{-1}\mathcal{C}_2)^{n-k+1} \\
&+& c_k(-\mathcal{C}_1^{-1}\mathcal{C}_2)^{n-k} + c_{k+1}(-\mathcal{C}_1^{-1}\mathcal{C}_2)^{n-k-1} + \dots c_{n}(-\mathcal{C}_1^{-1}\mathcal{C}_2)^{0}=0.
\end{eqnarray*}
Since $\mathcal{C}_1$ and $\mathcal{C}_2$ are invertible over the field $\Q(u_{i,j,k})$, the last equation is equivalent to
\begin{equation}\label{E:matint}
c_0(-\mathcal{C}_1^{-1}\mathcal{C}_2)^k + c_1(-\mathcal{C}_1^{-1}\mathcal{C}_2)^{k-1} + \dots c_{k-1}(-\mathcal{C}_1^{-1}\mathcal{C}_2)^{} = -(c_{k+1}(-\mathcal{C}_2^{-1}\mathcal{C}_1)^{0} + \dots c_{n}(-\mathcal{C}_2^{-1}\mathcal{C}_1)^{n-k}).
\end{equation}
If we view the matrix entries of both sides of Equation~\eqref{E:matint} as reduced ratios of elements of the UFD $\Lambda$, the denominator of the
left hand side can only involve $u_{1jk}$ and the denominator of the right hand side can only involve $u_{2jk}$.  Thus,
the matrices $c_0(-\mathcal{C}_1^{-1}\mathcal{C}_2)^k + c_1(-\mathcal{C}_1^{-1}\mathcal{C}_2)^{k-1} + \dots c_{k-1}(-\mathcal{C}_1^{-1}\mathcal{C}_2)^{}$
 must have all their entries in $\Lambda$.
\end{proof}

By definition, the $\Lambda$-algebra $R_c$ has a basis as a $\Lambda$-module given by $1,\zeta_1,\dots, \zeta_{n-1}$, where
$\zeta_k= c_0 \theta^k + \dots +c_{k-1}\theta \in \Lambda'[\theta]/c(\theta,1)$.
Thus the  action of $\zeta_k$ on $N_\mathcal{C}'$ is given by a matrix whose coefficients are in $\Lambda$,
and so it restricts to an action on $N_\mathcal{C}$.  An analogous argument can be made for $M_\mathcal{C}$.  This construction
is clearly equivariant for the $G_\Lambda$ actions.

\subsubsection{Balancing Map}
Now we will construct a map of $R_c$-modules $M_\mathcal{C} \tensor_{R_c} N_\mathcal{C} \ra I_c$.
The matrix $\mathcal{C}_1$ gives us an $\Lambda$-module pairing on $M_\mathcal{C}$ and $N_\mathcal{C}$ into $\Lambda$ by
$\alpha\star\beta =\alpha \mathcal{C}_1 \beta$ for $\alpha\in M_\mathcal{C}$ and $\beta\in N_\mathcal{C}$.
In other words, the matrix $\mathcal{C}_1$
which acts on $N_\mathcal{C}$ on the left as a $\Lambda$-module, gives $n$
homomorphisms of $\Lambda$-modules from $N_\mathcal{C}$ into $\Lambda$, one for each row of $\mathcal{C}_1$, and we map the $i$th basis element $m_i$ of $M$ to the  
$\Lambda$-module homomorphism of $N_\mathcal{C}$ into $\Lambda$ given by the $i$th row of $\mathcal{C}_1$.  
This gives
us a $\Lambda$-module map from $M_\mathcal{C}$ into $\Hom_\Lambda(N_\mathcal{C},\Lambda)$
and we have $\Hom_\Lambda(N_\mathcal{C},\Lambda)\isom \Hom_{R_c}(N,J_c)$, 
by Proposition~\ref{P:Jstruc}.
We use the symbol $\circ$ to denote the resulting pairing of $M_\mathcal{C}$ and $N_\mathcal{C}$ into $J_c$.
Thus, $\check{\zeta}_{n-1}(\alpha\circ\beta)= \alpha\star\beta$.
This pairing in $J_c$ 
is clearly equivariant for the $G_\Lambda$ actions.

Now we will show that $\circ$ gives a map of $R_c$-modules
$M_\mathcal{C} \tensor_{R_c} N_\mathcal{C} \ra I_c$.  To see this, we extend our pairing $\circ$ to
$M_\mathcal{C}' \tensor_{\Lambda'} N_\mathcal{C}' \ra J_c'$, which is a $R_c'$ module map for the $R_c'$ action on $N_\mathcal{C}$.
In fact, we can show that $\circ$ factors through $M_\mathcal{C}' \tensor_{R_c'} N_\mathcal{C}'$, i.e.  $(\theta \alpha)\circ\beta=\alpha\circ ( \theta\beta)$
for $\alpha\in M_\mathcal{C}'$ and $\beta \in N_\mathcal{C}'$.
If we fix an $\alpha$ and let $\beta$ vary over the elements of $N_\mathcal{C}'$, the expressions $(\theta \alpha)\circ\beta=\alpha\circ ( \theta\beta)$ give 
two homomorphisms from $N_\mathcal{C}'$ into $J_c'$ and we can check if they are the same by taking $\check{\zeta}_{n-1}$.
Now, $\check{\zeta}_{n-1}((\theta \alpha)\circ\beta)=(\theta \alpha)\star\beta=\alpha  (-\mathcal{C}_2 \mathcal{C}_1^{-1}) \mathcal{C}_1
\beta$ and
$\check{\zeta}_{n-1}(\alpha\circ ( \theta\beta))=\alpha \mathcal{C}_1(- \mathcal{C}_1^{-1}\mathcal{C}_2) \beta$, so we see that 
$\circ$ gives a map of $R_c'$ modules 
$M_\mathcal{C}' \tensor_{R_c'} N_\mathcal{C}' \ra J_c'$
and thus our original $\circ$ is a map
of $R_c$-modules $M_\mathcal{C} \tensor_{R_c} N_\mathcal{C} \ra J_c$.

We have an inclusion of $R_c$ modules
$I_c \subset J_c$ (given in Section~\ref{S:binf}), 
and we will use Lemma~\ref{L:inI} to see that for all $\alpha\in M_\mathcal{C}$ and $\beta\in N_\mathcal{C}$, the element 
$\alpha\circ\beta$ is in $I_c$.  Fix an $\alpha\in M_\mathcal{C}$. Then by Lemma~\ref{L:inI}, 
$ \alpha\circ N_\mathcal{C}   \subset I_c$ if and only if $\check{\zeta}_{n-2}(\alpha\circ N_\mathcal{C})\subset \Lambda$.
By Corollary~\ref{C:thet} we see that $\check{\zeta}_{n-2}(\alpha\circ N_\mathcal{C})=\check{\zeta}_{n-1}(\alpha\circ (\theta N_\mathcal{C}) )=
  \alpha\star(\theta N_\mathcal{C})$.
However, we have seen that the pairing $ \alpha\star (\theta\beta)$ is given by the matrix $-\mathcal{C}_2$, and thus
$\check{\zeta}_{n-2}(\alpha\circ N_\mathcal{C})\subset \Lambda$.  Thus we have given a map of $R_c$-modules $M_\mathcal{C} \tensor_{R_c} N_\mathcal{C} \ra I_c$.
Note that we have defined $\circ$ so that if $m_j$ and $n_k$ are the chosen bases of $M_\mathcal{C}$ and $N_\mathcal{C}$ respectively,
\begin{equation}\label{E:gotAB}
\check{\zeta}_{n-1}(m_j\circ n_k)=u_{1jk} \quad \mbox{and} \quad -\check{\zeta}_{n-2}(m_j\circ n_k)=u_{2jk},
\end{equation}
which makes $M_\mathcal{C}$ and $N_\mathcal{C}$ a based balanced pair of modules for $c$.

\subsubsection{Back to $\Z$}
Now given $A\in \Z^2 \tensor \Z^n \tensor \Z^n$, to find the action of $R_f$ on $M$, we take the matrix
by which $\zeta_k$ acted on $M_\mathcal{C}$ above, and substitute $a_{i,j,k}$ for the $u_{i,j,k}$, and similarly for $N$.  Of course,
the conditions for this to be a ring action will be satisfied since they are satisfied formally.  Also,
we have a map of $\Z$-modules $M \tesnor_\Z N \ra I_f$ given by specializing the formulas from the last section, and we can see that
this factors through a map of $R_f$-modules $M \tesnor_{R_f} N \ra I_f$ because the conditions for the factorization
and for the map to respect $R_f$-module structure are
satisfied formally.  Let $\psi(A)=(M,N)$.

\subsection{Proof of Theorem~\ref{T:bijZ} when $f_0\ne0$}\label{SS:proofZ}
Now we prove Theorem~\ref{T:bijZ} by showing that $\phi$ and $\psi$ are inverse constructions.
Suppose we have $A\in \Z^2 \tensor \Z^n\tensor \Z^n$.  Let $\psi(A)=(M,N)$, and let $A'=\phi(M,N)$.  By Equation~\eqref{E:gotAB},
we have that $A'=A$. Now, suppose we have $(M,N)$, a based balanced pair of modules for $f$, and $\phi(M,N)=A$ and
$\psi(A)=(M',N')$.  We first check that the action of $R_f$ is the same on $M$ and $M'$
(and $N$ and $N'$), and then we will check that the balancing maps agree.
We assume that $f_0\ne 0$. In this case, we may invert $f_0$ as in Section~\ref{SS:algZ},
and obtain $\Z[\theta]/F(\theta,1)$-modules $M_{f_0}$ and $N_{f_0}$.

\begin{proposition}\label{P:howact}
 If we write elements of $M_{f_0}$ as row vectors and elements of 
$N_{f_0}$ as column vectors, then
$\theta$ acts by $-A_2A_1^{-1}$ on the right  on $M_{f_0}$ and
$\theta$ acts by $-A_1^{-1}A_2$ on the left on $N_{f_0}$.
\end{proposition}
\begin{proof}
We let the map $M \tensor_\Z N \ra I_f$ be denoted by $\circ$.
We define $\alpha\star \beta$ to be $\check{\zeta}_{n-1}(\alpha \circ \beta)$.
We fix a non-zero $\alpha\in M_{f_0}$ and suppose for the sake of contradiction that $\alpha \star N_{f_0}=0$. 
Then $\alpha \circ N_{f_0}=0$ by Proposition~\ref{P:Jstruc}, and thus $\alpha A_1=\alpha A_2=0$.  Thus 
$\alpha$ is in the left kernel of $A_1x_1+A_2x_2$ for formal $x_i$ and so we obtain
$f=\Det(A_1x_1+A_2x_2)=0$, a contradiction.  Therefore, if  $\alpha \circ N_{f_0}=0$, then $\alpha=0$.
We have$(\theta\alpha)\star\beta=\check{\zeta}_{n-2}(\alpha\circ\beta)$
by Corollary~\ref{C:thet}
 and $\check{\zeta}_{n-2}(\alpha\circ\beta)=\alpha (-A_2) \beta =(\alpha (-A_2 A_1^{-1}))  A_1\beta$.
We conclude that 
$\theta\alpha=\alpha (-A_2 A_1^{-1})$.  A similar argument can be made for $N_{f_0}$.
\end{proof}

This proposition shows that the pairs of modules $(M,N)$ and $(M',N')$ have the same $R_f$ action.  We know that the map
$\Hom_{R_f} (M \tensor_{R_f} N, I )\ra \Hom_\Z (M \tensor_{\Z} N, V)$ is injective (from Proposition~\ref{P:Istruc}), and thus since 
$\phi(M,N)=A$ and $\phi(M',N')=\phi(\psi(A))=A$, we see that $(M,N)$ and $(M',N')$ have the same balancing map.
Therefore, we have proven Theorem~\ref{T:bijZ} when $f_0\ne 0$.  We will finish the proof at the beginning of the next section, by reducing to this case.

\subsection{$\GL(V)$ invariance}\label{S:GLinv}
Let $V=\Z^2$, and we have that $\GL_2(\Z)=\GL(V)$ acts on $V\tesnor \Z^n \tensor \Z^n$ and also on binary $n$-ic
forms in $\Sym^n V$.  The determinant map is equivariant for these actions.  Let $g\in \GL(V)$, so that
$g(A)=A'$ and $g(f)=f'$.  Then we have isomorphisms $R_f\isom R_f'$, and $I_f\isom I_f'$
(e.g. by the geometric construction of \cite[Section 2.3]{Bf}). 
In fact, $g$ also gives a map $V\ra V$ such that the diagram
$$
\begin{CD}
 I_f @>g>> I_f' \\
 @ VVV @ VVV \\
 V @>>g> V \\
\end{CD}
$$
commutes.

More concretely, consider $g=\left(\begin{smallmatrix} a & b \\ c& d \end{smallmatrix}\right)\in\GL_2(\Z)$.
If $A=(A_1,A_2)$, then $g(A)=A'=(A_1',A_2')=(aA_1+bA_2,cA_2+dA_2)$.  If $f=F(x,y)$, then $g(f)=f'=F(ax+cy,bx+dy)$.
Write $f'=F'(x,y)$.  If $\theta$ is a root of $F(x,1)$, then $\theta'=\frac{d\theta-c}{-b\theta+a}$ is a root of $F'(x,1)$.
This induces the map $R_f'\isom R_f$.  Note that $\theta=\frac{a\theta'+c}{b\theta'+d}.$
We can view $I_f$ and $I_f'$ as fractional ideals in the same $\Q$-algebra.
They are given as fractional ideals of different $\Q$-algebras, $Q_f$ and $Q_f'$ respectively,  in Section~\ref{S:binf},
but the map $\theta'\mapsto\frac{d\theta-c}{-b\theta+a}$ gives an isomorphism of those $\Q$-algebras.
Then the map $I_f\isom I_f'$ is given by
$$
I_f \ra Q_f \isom Q_f' \stackrel{\times (b\theta'+d)^{n-3}}{\longrightarrow} Qf' \supset I_f'.
$$
Note that $\frac{1}{b\theta'+d}=\frac{a-b\theta}{ad-bc}.$

Viewing $\check{\zeta}_{n-1}$,$\check{\zeta}_{n-2}$ and $\check{\zeta}'_{n-1}$,$\check{\zeta}'_{n-2}$ as maps of 
 $I_f$ and $I_f'$, respectively, we have that
$I_f\isom I_f'$ induces 
$$
\check{\zeta}'_{n-1}\mapsto a\check{\zeta}_{n-1}-b\check{\zeta}_{n-2} \text{  and  } 
-\check{\zeta}'_{n-2}\mapsto \check{\zeta}_{n-1}-d\check{\zeta}_{n-2},
$$
which exactly gives that our construction of $(A_1,A_2)$ from $\check{\zeta}_{n-1}$,$-\check{\zeta}_{n-2}$ is equivariant.
We can check this on a generating set of $\GL_2(\Z)$, though it also follows from \cite[Proposition 3.3]{Bf}.
If we write an element $v$ of $V$ as a column vector, then $g$ acts on $V$ by the standard left action.
In the map from $I_f\ra V$, an element $\alpha\in I_f$ maps to 
$\left[\begin{smallmatrix}
        \check{\zeta}'_{n-1}(\alpha)\\
 -\check{\zeta}'_{n-2}(\alpha)
       \end{smallmatrix}
 \right] .$

Our constructions of $M$, $N$, and the balancing map are equivariant under this $\GL(V)$ action.
More precisely, under the identifications $R_f\isom R_f'$ and $I_f\isom I_f'$ and the map $V \stackrel{g}{\ra} V$, the based modules and balancing map
we obtain from $A$ are the same as the based modules and balancing map we obtain from $A'$.
(This can easily be checked on a basis of $\GL_2(\Z)$, or alternatively, it follows from
the geometric versions of the constructions in Section~\ref{SS:geomZ}.
For example, if $\theta$ acts like $-A_2A_1^{-1}$ then $\theta'=\frac{d\theta-c}{-b\theta+a}$
will act like $(-dA_2A_1^{-1}-c)(bA_2A_1^{-1}+a)^{-1}=-(A_2')(A_1')^{-1}$.)
Thus, to check that the $R_f$ action and balancing map on pairs $M,N$ and $M',N'$ agree, we can check after a
$\GL_2(\Z)$ action on $f$ so that $f_0\ne 0$ (as long as $f\ne0$).  This proves Theorem~\ref{T:bijZ}.

\section{Symmetric tensors}\label{S:sym}

In the map
$$
\bij{based balanced pairs $(M,N)$ of modules for $f$}{$A\in \Z^2 \tensor \Z^n \tensor \Z^n$ with $\Det(A)=f$}.
$$
of Theorem~\ref{T:bijZ}, it is easy to see from the construction that pairs where $M$ and $N$
are the same based module exactly correspond to $A$ such that $A_1$ and $A_2$ are symmetric matrices.

\begin{definition}
A \emph{self balanced module} for a non-zero form $f$ is an 
 $R_f$-module $M$, that is a free rank $n$ $\Z$-module,
and a map of $R_f$-modules $M \tensor_{R_f} M \ra I_f$, such that when the composition
$M \tensor_\Z M \ra M \tensor_{R_f} M \ra I_f \ra V$ is written as a pair of matrices $A_1$ and $A_2$ (using a single choice of basis for $M$),
 we have $\Det(A_1 x_1 + A_2 x_2)= f$.
\end{definition}

We easily conclude the following.
\begin{theorem}\label{T:gensym}
For every non-zero binary $n$-ic form $f$ with coefficients in $\Z$, there is a bijection
$$
\bij{isomorphism classes self balanced modules $M$ for $f$}{$\GL_n(\Z)$ classes of $A\in \Z^2 \tensor \Sym_2 \Z^n $ with $\Det(A)=f$},
$$
where $\Sym_2 \Z^n$ are symmetric $n$ by $n$ matrices, and the action of $g\in\GL_n(\Z)$ is by multiplication on the left by $g$ and right by $g^t$.
\end{theorem}

\section{Equivalent formulations of the balancing condition}\label{S:balance}

In order to prove Theorems~\ref{T:specbijI} and \ref{T:primsym} in the Introduction, we will show that for primitive
forms, modules that appear in balanced pairs have unique balance partners, and that for non-degenerate forms, modules that appear
in balanced pairs can be realized as fractional ideals.
First, we will see an equivalent formulation of the definition of balanced.

We define a \emph{characteristic $R_f$-module} $M$ to be a $R_f$-module $M$ which is a free rank $n$ $\Z$-module such that for any element $\zeta \in R_f$
the action of $\zeta$ on $M$ viewed as a $\Z$-module has the same characteristic polynomial as the action
of $\zeta$ on $R_f$ (by multiplication) viewed as a $\Z$-module.  
Fractional ideals of $R_f$ are characteristic modules of $R_f$.

Given two based modules $M$ and $N$ with bases $\alpha_i$ and $\beta_i$ respectively such that $M \subset N$,
the \emph{index} $[N:M]$ is the absolute value of the determinant of the matrix $Q$ with entries in $\Z$ such that
$[\begin{matrix} \alpha_1 & \alpha_2 &\dots & \alpha_n\end{matrix}]=[\begin{matrix} \beta_1 & \beta_2 &\dots & \beta_n\end{matrix}]
\cdot Q$. 
Now we will see that our condition of balanced is equivalent to $M$ characteristic and an index condition on the map
$M\tensor_{R_f} N\ra I_f$.

\begin{proposition}\label{P:eqb}
Consider a non-zero binary form $f$ over $\Z$, and two $R_f$ modules $M$ and $N$, with a $R_f$-module map $M\tesnor_{R_f} N \ra I_f$, such
that $M$ and $N$ are both free rank $n$ $\Z$-modules.  Then this data gives a balanced pair of modules
for $f$ if and only if $M \subset \Hom_{R_f} (N,I_f)$, and $[\Hom_{R_f} (N,J_f):M]=[J_f:I_f]$ 
(with any inclusion of $I_f$ in $J_f$ as $R_f$-modules), and 
either $M$ or $N$ is characteristic.

\end{proposition}

The equality of indexes does not depend on the choice of inclusion of $I_f$ in $J_f$, because any two inclusions differ by multiplication by a non-zero-divisor in $R_f\tensor_\Z \Q$.  This multiplies both $[\Hom_{R_f} (N,J_f):M]$ and $[J_f:I_f]$ by the absolute value of the norm of that element.

\begin{proof}
We can act by $\GL_2(\Z)$ so as to assume $f_0\ne0$.  Then, we use the inclusion of $I_f$ in $J_f$ given in Section~\ref{S:binf}
and see that $[J_f:I_f]=f_0$.  

\begin{lemma}
Suppose we have two $R_f$ modules $M$ and $N$, with a map $M\tesnor_{R_f} N \ra I_f$, such
that $M$ and $N$ are both free rank $n$ $\Z$-modules, either $M$ or $N$ is characteristic, and $M \subset \Hom_{R_f} (N,I_f)$ such that $[\Hom_{R_f} (N,J_f):M]=f_0$.  Let $A$, as usual, denote the map $M\tensor_\Z N \ra V$.  
 We write elements of $M$ as row vectors with entries in $\Z$.  Then $\theta$ acts on $M'=M_{f_0}$ by $-A_2A_1^{-1}$ on the right, and
$\theta$ acts on $N'=N_{f_0}$ by $-A_1^{-1}A_2$ on the left.
Also, $\Det(A_1x_1+A_2x_2)=f$.
\end{lemma}

\begin{proof}
We define $m\star n$ to be $\check{\zeta}_{n-1}(m\circ n)$.
By Proposition~\ref{P:Jstruc}, we see that $\star$ is faithful
for both $M$ and $N$ if and only if $\circ$ is.  
We see that $\circ$ is faithful on $M$ since
the index of $M$ in $\Hom_{R_f}(N,J_f)$ is not zero.
If there were
an $n\in N$ such that $ M\circ n =0$ for all $n$, then inverting ${f_0}$
we would find that $M_{f_0}\star n=0$ but
$M_{f_0}=\Hom_{\Z_{f_0}}(N_{f_0},\Z_{f_0})$, and we obtain a
 contradiction since $N_{f_0}$ is a free $\Z_{f_0}$-module and thus
there is some homomorphism from $N_{f_0}$ to $\Z_{f_0}$ which is non-zero on $n$.
 So, we fix an $\alpha\in M$ and
let $\beta$ vary in $N$.  Then $\alpha\theta\star\beta=\check{\zeta}_{n-2}(\alpha\circ\beta)=
-\alpha A_2 \beta=-\alpha A_2 A_1^{-1} A_1 \beta=(-\alpha A_2 A_1^{-1}) \star \beta$.  We similarly obtain the action of $\theta$ on $N$.

We have that $|\Det(A_1)|$ is the index of $M$ in $\Hom_\Z(N,\Z)$ or $\Hom_{R_f}(N,J_f)$, which is
the index of $I_f$ in $J_f$, i.e. $|f_0|$.   Since $M$ is characteristic, we know that $\theta$ and thus $-A_2A_1^{-1}$
acts with characteristic polynomial $F(t,1)/f_0$ on $M_{f_0}$.  It follows
that $\Det(A_1x_1+A_2x_2)=\pm f$.  We can argue similarly if $N$ is characteristic.
\end{proof}

Now, suppose that $M,N$ are balanced.  Then we know that $M,N$ are constructed from an element $A\in Z^2\tensor \Z^n \tensor \Z^n$
such that $\Det(A)=f$.  
We can see from the construction of the action of $\theta$ on $M_{f_0}$ that $M$ is characteristic.  
We have a map $M \ra \Hom_{R_f} (N,I_f)$, and composition with $I_f\subset J_f$ gives
$M \ra \Hom_{R_f} (N,J_f)=\Hom_\Z(N,\Z)$.  The map $M \ra\Hom_\Z(N,\Z)$ is given by $A_1$, and thus
$[\Hom_{R_f} (N,J):M]=|f_0|$, which implies $M \subset \Hom_{R_f} (N,J_f)$, and thus the map $M \ra \Hom_{R_f} (N,I_f)$ is injective as well.
\end{proof}

\begin{corollary}\label{C:Ninv}
For a non-zero form $f$, if $N$ is a finitely generated invertible module for $R_f$, then
there exists a unique balancing partner $M$ for $N$.
\end{corollary}
\begin{proof}
If $N$ is a finitely generated invertible module, then $N$ can be realized an an invertible fractional ideal of $R_f$
\cite[II.5.7, Proposition 12]{B}.
 Then $\Hom_{R_f} (N,J_f)=N^{-1}J_f$ and $\Hom_{R_f} (N,I_f)=N^{-1}I_f$.
In that case, $[\Hom_{R_f} (N,J_f):\Hom_{R_f} (N,I_f)]=[J_f:I_f]$ and for $M$ to be balanced with $N$ it is necessary and sufficient that $M=\Hom_{R_f} (N,I_f)$.
\end{proof}

\subsection{Non-degenerate forms}\label{SS:nodeg}
When $f$ is a non-degenerate binary $n$-ic form over $\Z$,
we have the following. 
\begin{proposition}\label{P:charfrac}
If $f$ is a non-degenerate binary $n$-ic form over $\Z$ (i.e. $\disc(f)\ne0$), then all characteristic modules 
can be realized as fractional ideals.  This gives a bijection between isomorphism classes of 
characteristic $R_f$-modules and fractional ideal classes of $R_f$.
\end{proposition}
\begin{proof}
We assume $f_0\ne 0$ by an action of $\GL_2(\Z)$.  Then, we see that we can put the action of $\theta$
on a characteristic module $M$ in rational normal form over $\Q$, and since it acts with the same
separable characteristic polynomial as the action of $\theta$ on $R_f\tensor_\Z \Q$, in rational normal form these
actions must be the same.  Thus, we can view $M$ as a $\Z$-submodule of $R_f\tensor_\Z \Q$, or a fractional ideal.
Clearly two fractional ideals in the same class give isomorphic modules.  Moreover, a module homomorphism between two fractional ideals
$I_1\ra I_2$ sends $q\in \Q\cap I_1$ to some element $k\in I_2$, and since the map is an $R_f$-module map, we see that
it is multiplication by $k/q$.
\end{proof}

For a fractional ideal $M$, let $|M|$ denote the norm of $M$, given by the index
$[R_f:M]$, which can be defined even if $M$ is not a submodule of $R_f$, since they sit in a common $\Q$-vector space.  Then, we can reformulate the balancing condition in terms of norms.  This is the version of balanced used in 
\cite{HCL1} and \cite{HCL2}.  

\begin{theorem}\label{T:sameBhar}
For non-degenerate $f$, we have a bijection
$$
\bij{isomorphism classes of balanced pairs $(M,N)$ of modules for $f$}{classes of $(M,N)$ where $M$ and $N$ are fractional $R_f$-ideals, $MN\subset I_f$ and $|M||N|=|I_f|$},
$$
where $(M,N)$ and $(M_1,N_1)$ are in the same class if
$M_1=\lambda M$ and $N_1=\lambda^{-1} N$ for some invertible element 
$\lambda\in R_f\tensor_\Z \Q$.
\end{theorem}

\begin{proof}
All modules that appear in balancing pairs are characteristic by Proposition~\ref{P:eqb}, and thus can be realized as fractional ideals.  For a balanced pair $M,N$ of modules, we can take any fractional ideal representative of $M$, but then we choose the
unique representative of $N$ such that the map $M\tensor N \ra I_f$
is just given by $M\tesnor N \ra MN \subset R_f\tesnor_\Z \Q$ with
image landing in $R_f$.   If $M$ and $N$ are fractional ideals of $R_f$, a map $M\tensor_{R_f} N \ra I$ factors through $MN$.

We now argue that $\tau: MN \ra I_f$ is injective.  As usual, we assume $f_0\ne 0$ by a $\GL_2(\Z)$ action if necessary.
We can detect the injectivity after tensoring with $\Q$ because $\Q$ is a flat $\Z$-module.  Over $\Q$ we have that $MN$ is at least rank $n$ because it contains $N$ and thus is rank $n$.  We can take $\theta^k$ as a basis of $MN$, and we see where they map to in $I_f=J_f=\Hom(R,\Q)$. Then $\tau(\theta^k)$ is the map in $\Hom(R,\Q)$ that sends $\zeta_i$ to $\check{\zeta}_{n-1} (\zeta_i \theta^k)$. By Proposition~\ref{P:getcoeff}, we have 
$$\check{\zeta}_{n-1} (\zeta_i \theta^k)=
\begin{cases}
1 &\text{ if $k+i=n-1$ and $i>0$},\\
1/f_0 &\text{ if $k=n-1$ and $i=0$},\\
0 &\text{ otherwise}.
\end{cases}
$$
Thus, we see that $\tau(MN)=\Hom(R,\Q)$, when working over $\Q$, and therefore over $\Z$ we have that $MN \ra I_f$ is injective. 

A map $MN\ra I_f$ is just multiplication by some element of
$R_f\tesnor_\Z \Q$.  The element is not a zero-divisor since
$MN\ra I_f$ is injective, and thus it is invertible in $R_f\tensor_\Z \Q$.  
We can choose that element to be 1 by taking
a different representative for $N$ in its ideal class.  If 
we had chosen a different representative for $M$, this would change the class of $(M,N)$.

Suppose we have a balanced pair $(M,N)$ realized as ideal classes with
$MN\subset I$.  We will show that the index condition for balanced is equivalent to the norm condition in the above theorem.  

\begin{proposition}
Let $f$ be a non-zero form.  If $M$ and $N$ are fractional ideals of $R_f$ with $MN\subset I_f$, then $[\Hom_{R_f} (N,J_f):M]=[J_f:I_f]$
if and only if  and $|M||N|=|I_f|$. 
\end{proposition}
\begin{proof}
We can act by $\GL_2(\Z)$ so as to assume $f_0\ne0$.
We claim that $|M||N|$ is the product of $|J_f|$ with the determinant of the pairing
$\zeta_{n-1}(mn)$.
When $M=R_f$ and $N=J_f$, we see from Proposition~\ref{P:getcoeff} that the determinant of the pairing is 
$1$, and thus the claim is true.  If we change $\Q$-bases from $R_f, J_f$ to $M,N$, we change the determinant of the pairing
by $N(M)N(N)/N(J_f)$ and thus the determinant of the pairing $\zeta_{n-1}(mn)$ is $|M||N|/|J_f|$.

The index of $M$ in $\Hom_{R_f}(N,J_f)$ is just the index of $M$ in $\Hom_\Z(N,\Z)$, which is giving by the pairing
$\zeta_{n-1}(mn)$.  Thus $[\Hom_{R_f}(N,J_f):M]=|M||N|/|J_f|$.  We see that $[\Hom_{R_f} (N,J_f):M]=[J_f:I_f]$ if and only if
$|M||N|=|I_f|$.
\end{proof}

The theorem now follows from the above proposition.

\end{proof}

For symmetric tensors, we can make a similar argument to prove the following.

\begin{theorem}
For non-degenerate $f$, we have a bijection
$$
\bij{isomorphism classes of self balanced of modules $M$ for $f$}{classes of $(M,k)$ where $M$ is a fractional $R_f$-ideal,
$k$ is an invertible element of $R_f\tensor_\Z \Q$, and $M^2 k\subset I_f$ and $|M|^2|(k)|=|I_f|$},
$$
where $(M,k)$ and $(M_1,k_1)$ are in the same class if
$M_1=\lambda M$ and $k_1=\lambda^{-2} k$ for some invertible element 
$\lambda\in R_f\tensor_\Z \Q$.
\end{theorem}

\subsection{Primitive forms}\label{SS:prim}
If $f$ is primitive them $I_f$ and $J_f$ are invertible $R_f$ modules
\cite[Proposition 2.1]{Bf}.  For general non-zero forms, we saw in Corollary~\ref{C:Ninv} that invertible ideals have unique balancing partners.  For primitive $f$ we have the following.

\begin{proposition}\label{P:primu}
If $f$ is a primitive non-degenerate form, and $N$ is a characteristic $R_f$-module, then there exists a unique balancing partner $M$ for $N$
(i.e. an $R_f$-module $M$ and map $M\tensor_{R_f} N \ra I_f$ 
that gives a balanced pair).
\end{proposition}
\begin{proof}
In this case, we see that $[\Hom_{R_f} (N,J_f):\Hom_{R_f} (N,I_f)]=[J_f:I_f]$.  This is because
$\Hom_{R_f} (N,J_f)$ and $\Hom_{R_f} (N,I_f)$ are naturally realized as fractional $R_f$ ideals.  
Then we see that $\Hom_{R_f} (N,J_f)J_f^{-1} I_f\subset \Hom_{R_f} (N,I_f)$ and
$\Hom_{R_f} (N,I_f)I_f^{-1} J_f\subset \Hom_{R_f} (N,J_f)$.  Thus $\Hom_{R_f} (N,J_f)=\Hom_{R_f} (N,I_f)I_f^{-1} J_f$,
and $[\Hom_{R_f} (N,J_f):\Hom_{R_f} (N,I_f)]$ is the norm of $J_f I_f^{-1}$, which is $[J_f:I_f]$.
Then, for $M$ to be balanced with $N$ it is necessary and sufficient that $M=\Hom_{R_f} (N,I_f)$.
\end{proof}

Theorem~\ref{T:specbijI} now follows from Propositions~\ref{P:primu} and 
\ref{P:charfrac} and Theorem~\ref{T:bijZ}.  
We can also apply Proposition~\ref{P:primu} to symmetric tensors.
\begin{theorem}\label{T:primequiv}
For non-degenerate primitive $f$, we have a bijection
$$
\bij{isomorphism classes of self balanced of modules $M$ for $f$}{classes of $(M,k)$ where $M$ is a fractional $R_f$-ideal,
$k$ is an invertible element of $R_f\tensor_\Z \Q$, and $M=(I_fk:M)$},
$$
where $(M,k)$ and $(M_1,k_1)$ are in the same class if
$M_1=\lambda M$ and $k_1=\lambda^{2} k$ for some invertible element 
$\lambda\in R_f\tensor_\Z \Q$, and $(I_fk:M)$ is the fractional ideal of elements $x$ such that
$xM\subset I_fk$.
\end{theorem}

\section{Main theorem over an arbitrary base}\label{S:arbbase}
The proof of Theorem~\ref{T:bijZ} works over an arbitrary base with some modifications.  Let $S$ be a scheme.  We
consider binary $n$-ic forms with coefficients in $\OS$, i.e. $f_0x_1^n +f_1 x_1^{n-1}x_2 +\dots +f_nx_2^n$ with $f_i\in \OS$.
We say such a form is a \emph{zero-divisor} if it is a zero divisor in the $\OS$-algebra $\OS[x,y]$, which means
that for some open $\mathcal{U}$ of $S$, that $f$ is zero divisor in $\OS[x,y](\mathcal{U})$.
We can construct an $\OS$-module $R_f$ 
and an $R_f$-module $I_f$ by using the construction
of Section~\ref{S:binf} for the universal form over $\Z[f_0,\dots,f_n]$ and then pulling back to $S$ with 
the map given by our desired form.  From this construction, we inherit a map of $\OS$-modules
$I_f \ra V=\OS^{ 2}$.

\begin{definition}
 A \emph{based balanced pair of modules} for $f$ is a 
pair of $R_f$-modules $M$ and $N$, a choice of basis
$M \isom \OS^n$ and $N \isom \OS^n$, and a map of $R_f$-modules
$M \tensor_{R_f} N \ra I_f$, such that when the composition
$M \tensor_\OS N \ra M \tensor_{R_f} N \ra I_f \ra V$ is written as a pair of matrices
$A_1$ and $A_2$, we have $\Det(A_1 x_1 + A_2 x_2)=f$.
 A \emph{balanced free pair of modules} for $f$ is a 
pair of $R_f$-modules $M$ and $N$, each a free rank $n$ $\OS$-module, and a map of $R_f$-modules
$M \tensor_{R_f} N \ra I_f$, such that when the composition
$M \tensor_\OS N \ra M \tensor_{R_f} N \ra I_f \ra V$ is written as a pair of matrices $A_1$ and $A_2$,
we have $\Det(A)=fu$, where $u$ is a unit in $\OS$. 
Given a balanced pair of modules for a non-zero-divisor form $f$, there is
a unique choice of generator $\chi$ of $\wn M \tensor \wn N$
such that $\Det(A_1 x_1 + A_2 x_2)=f$
when constructing $A$ with bases of $M$ and $N$ that give $\chi\in \wn M \tensor \wn N$, we obtain  $\Det(A)=f$.

\end{definition}

\begin{theorem}
For every non-zero-divisor binary $n$-ic form $f$ with coefficients in $\OS$, there is a bijection
$$
\bij{based balanced pairs $(M,N)$ of modules for $f$}{$A\in \OS^2 \tensor \OS^n \tensor \OS^n$ with $\Det(A)=f$}.
$$
Let $G$ be the subgroup of $\GL_n(\OS)\times\GL_n(\OS)$ of elements $(g_1,g_2)$ such that
$\Det(g_1)\Det(g_2)=1$.  Then, $G$ acts equivariantly in the above bijection
(acting of the bases of $M$ and $N$), and we obtain a bijection
$$
\bij{isomorphism classes of balanced free pairs $(M,N)$ of modules for $f$}{$G$-classes of $A\in \OS^2 \tensor \OS^n \tensor \OS^n$ with $\Det(A)=f$}.
$$
\end{theorem}
\begin{proof}
To construct a based balanced pair of modules from $A\in \OS^2 \tensor \OS^n \tensor \OS^n$, we can simply pullback the 
construction from the universal tensor (and we call this construction $\psi$).
Again, the balancing  map composed with $I_f\ra V$ gives the construction $\phi$ of an element of
$\OS^2 \tensor \OS^n \tensor \OS^n$ from a based balanced pair.
Now, suppose we have $(M,N)$, a based balanced pair of modules for $f$, and $\phi(M,N)=A$ and
$\psi(A)=(M',N')$.  We need to check that the action of $R_f$ is the same on $M$ and $M'$
(and $N$ and $N'$), and that the balancing maps agree.  It suffices to check this everywhere locally over $S$, and so
we can assume that $S$ is affine, and $S=\Spec B$.  Then, if suffices to check in a larger ring, so we
let $E$ be the ring obtained from inverting all the non-zero-divisors in $B[x,y]$.

 We have that $B[x,y]\subset E$.
We see that $x$ is not a zero divisor in
$B[x,y]$, because $xg=0$ implies that the leading coefficient of $g$ is 0.  We consider
$G(t_1,t_2)=F(xt_1,yt_1+\frac{1}{x}t_2)$.  This is a binary $n$-ic form in variables $t_i$ with coefficients in $E$.
Over $E$ we see it is a $\GL_2(E)$ transformation of $f$.  
We have that $G(1,0)=F(x,y)$, and thus $f$ is the leading coefficient of the new form.
However, $f$ has an inverse in $E$ and thus is not a zero divisor.  By the $\GL_2$ invariance of our constructions,
we can reduce to checking in the case where $f_0$ is not a zero divisor.  In this case we can prove
Proposition~\ref{P:howact} just as in the case of $\Z$.
\end{proof}

In fact, we can consider a completely general binary $n$-ic form
over $S$ given by a locally free rank $2$ $\OS$-module $V$, a locally free rank 1 $\OS$-module $L$, and a global section
$f\in \Sym^n V \tensor L$.  
We say a form $f$ is a \emph{zero-divisor} if it is a zero divisor on any open
$\mathcal{U}$ of $S$ on which $V$ and $L$ are free (and in this case the notion of zero-divisor is defined above).
When $f$ is not a zero-divisor, we have an associated $\OS$-algebra $R_f:=\pi_*(\O_{T_f})$, where
$T_f$ is the subscheme of $\P(V)$ cut out by $f$, and $\pi: T_f \ra S$.
We have line bundles $\O_{T_f}(k)$ on $T_f$ pulled back from $\O(k)$ on $\P(V)$, and we define
$I_f:= \pi_*(\O_{T_f}(n-3))\tensor (\wedge^2 V)^{\tensor 2} \tensor L $, which has an action 
of $R_f$ through the first factor.
Then we have a natural map $I_f\ra V^* \tesnor \wedge^2 V\isom V$ as in \cite[Equations (3.8) and (3.9)]{Bf}.  
Locally on $S$ where $V$ and $L$ are free, the constructions of $R_f$ and $I_f$ pullback from the constructions 
of $R_f$ and $I_f$ for the universal form over $\Z[f_0,\dots,f_n]$ as given in Section~\ref{S:binf} and used at the start of this section.
See \cite[Section 3]{Bf} for more details of these constructions.

We now consider $A\in V\tensor U\tensor W$, where $U$ and $W$ are locally free rank $n$ $\OS$-modules with
an orientation isomorphism $\wn U \tensor \wn W\isom L$.  An isomorphism between
$A\in V\tensor U\tensor W$ and $A'\in V\tensor U'\tensor W'$ is given by isomorphisms $U\isom U'$ and
$W \isom W'$ that take $A$ to $A'$ and respect the orientations.  We have the determinant 
 $\Det(A)\in\Sym^n V\tensor \wn U \tensor \wn V\isom \Sym^n V \tensor L$. 
(see \cite[Section 8.2]{Quartic} for the details of this kind of construction).
 Given a scheme $S$ and a locally free $\OS$-module $U$, we let $U^*$ denote the $\OS$-module $\sHom_{\OS} (U,\OS)$,
even if $U$ is also a module for another sheaf of algebras.

\begin{definition}
 A \emph{balanced pair of modules} for a non-zero-divisor $f$ is a 
pair of $R_f$-modules $M$ and $N$, each a locally free rank $n$ $\OS$-module such that $\wn M \tensor \wn N\isom L^*$, and a map of $R_f$-modules $M \tensor_{R_f} N \ra I_f$, such that when the composition
$M \tensor_\OS N \ra M \tensor_{R_f} N \ra I_f \ra V$ is written as $A\in M^* \tesnor N^* \tensor V$
 the image of $\Det(A)$ in $\Sym^n V \tesnor L$ (via an isomorphism $\wn M^* \tensor \wn N^*\isom L$) is $fu$, where $u$ is a unit in $\OS$. 
When $f$ is a non-zero-divisor, given that $\wn M \tensor \wn N\isom L^*$, there is a unique choice of isomorphism
 so that the image of $\Det(A)$ is $f$. 
\end{definition}

\begin{theorem}\label{T:ARB}
For every non-zero-divisor binary $n$-ic form $f\in \Sym^n V \tensor L$, there is a bijection
$$
\bij{isomorphism classes of balanced pairs $(M,N)$ of modules for $f$}{isomorphism classes of $A\in V\tensor U\tensor W$, where $U$ and $W$ are locally free rank $n$ $\OS$-modules with
an orientation isomorphism $\wn U \tensor \wn W\isom L$, and $\Det(A)=f$}.
$$
\end{theorem}
\begin{proof}
From $A$, we can construct $R_f$ modules from $U^*$ and $W^*$ by giving the $R_f$ action locally where
$U$ and $W$ are free and we can chose bases, and then seeing that it is invariant under change of basis by elements of $\GL_n\times\GL_n$ that preserve the orientation.  Similarly, we can construct the balancing map. 
Again, the construction of $A$ from a balanced pair of modules just combines the balancing map with $I_f\ra V$.
To see that these constructions are inverse, it suffices to check locally on $S$, where we can assume $V$, $U$, and $W$ are free.
\end{proof}

As when working over $\Z$, we can also get a version on the theorem for symmetric tensors.

\begin{definition}
 A \emph{self balanced module} for a non-zero-divisor 
 $f$ is an $R_f$-modules $M$, that is a locally free rank $n$ $\OS$-module and such that
$(\wn M)^{\tensor 2}$ is isomorphic to $L^*$, and a map of $R_f$-modules
$M \tensor_{R_f} N \ra I_f$, such that when the composition
$M \tensor_\OS N \ra M \tensor_{R_f} N \ra I_f \ra V$ is written as $A\in M^* \tesnor N^* \tensor V$
 we have $\Det(A)=fu$, where $u$ is a unit in $\OS$.
\end{definition}

\begin{theorem}
For every non-zero-divisor binary $n$-ic form $f\in \Sym^n V \tensor L$, there is a bijection
$$
\bij{isomorphism classes of self balanced modules $M$ for $f$}{isomorphism classes of $A\in V\tensor \Sym_2 U$, where $U$  is a locally free rank $n$ $\OS$-module with
an orientation isomorphism $(\wn U )^{\tensor 2}\isom L$, and such that $\Det(A)=f$}.
$$
\end{theorem}

\section{Geometric construction of modules from the universal tensor}\label{SS:geomZ}
Just as we have given both concrete and geometric constructions of $R_f$, $I_f$, and $J_f$, in this section we will give 
geometric constructions of the $R_f$-modules $M$ and $N$ that were constructed explicitly above for the universal tensor. 

\begin{notation}
Let $B$ be a ring.  
 When we have an a matrix $M \in B^m \tensor B^n$, we can multiply $M$ by vectors in two ways.  We can multiply
$M$ by a length $m$ vector on the left, and we can multiply $M$ by a length $n$ column vector on the right.
When we have an element $A\in B^\ell \tensor B^m \tensor B^n$, we can multiply it by vectors in three different ways, and 
we realize that the ``on the left'' and ``on the right'' descriptions do not generalize appropriately for three dimensional tensors.
We will need a new language.  An element $A\in B^\ell \tensor B^m \tensor B^n$ is comprised of entries
$a_{i j k}$, with $1\leq i\leq \ell,$ $1\leq j\leq m$, and $1\leq k \leq n$.  We say that $a_{ijk}$ is the entry in the
 $i$th \emph{aisle}, $j$th \emph{row}, and $k$th \emph{column}.  Note that aisle, row, and column denote two dimensional
 submatrices, i.e. codimension one slices of $A$.   For $A\in B^2 \tensor B^n \tensor B^n$, the $n$ by $n$ matrix
we called $A_i$ above is the $i$th aisle of $A$.  

If we have a sequence $x_1,\dots,x_\ell$, we can
form it into a vector and combine it with $A\in B^\ell \tensor B^m \tensor B^n$ to get the $m$ by $n$ matrix we call $A(x,\cdot,\cdot)$ with $j, k$ entry
$\sum_i a_{ijk} x_i$.  The dots indicate that we have not also multiplied by vectors in the other situations. 
For example, for $A\in B^2 \tensor B^n \tensor B^n$, the matrix $A(x,\cdot,\cdot)$ is what we have previously
referred to as 
$A_1 x_1 + A_2 x_2$.
 Similarly, if
we have a sequence $y_1,\dots, y_m$, we can form a $2 \times n$ array $A(\cdot,y,\cdot)$ with $i,k$ entry
$\sum_j a_{ijk} y_j$.  We could call this array a matrix, but it is more convenient to continue to refer
to its aisles and columns.  If
we have a sequence $z_1,\dots, z_n$, we can form a $2 \times m$ array $A(\cdot,\cdot,z)$ with $i,j$ entry
(in the $i$th aisle and $j$th row)
$\sum_\ell a_{ijk} z_k$.
In fact, we will always use a $x$ variable in the first place, $y$ in the second place, and a $z$ in the third place, and thus we will use the short hand $A(x)$ for $A(x,\cdot,\cdot)$ and $A(y)$ for $A(\cdot,y,\cdot)$.  
We may refer to the $j,k$ entry of $A(x)$ by $A(x)_{j,k}$ and the $i,k$ entry of $A(y)$ by
$A(y)_{i,k}$.
We can also multiply $A$ by more than one vector at a time.  For example, $A(x,y,\cdot)$ (denoted by $A(x,y)$ for short)
is a length $n$ vector with $\ell$th entry $\sum_i \sum_j a_{ijk} x_i y_j$.

Given a 2-dimensional array $A$ with entries in some ring, we can form the ideal
$\M(A)$ of the determinants of its maximal minors.  For example,
for $A\in B^2 \tensor B^n \tensor B^n$, we have that
$\M(A(x))=(\Det(A_1x_1+A_2x_2))$.  We have previously called the subscheme of $\P^1_B$  defined by this
ideal $T_{\Det(A_1x_1+A_2x_2)}$.  Now, in order to emphasize certain symmetries, we say
that $\M(A(x))$ defines a subscheme $T_{A(x)} \sub \P^1_B$.  Analogously, we have a subscheme
$T_{A(y)} \sub \P^{n-1}_B$ cut out by the determinants of $2\times 2$ minors of $A(y)$.
\end{notation}



The scheme $T_{A(y)}$ has line bundles
$\O_{T_{A(y)}}(k)$ pulled back from $\O(k)$ on $\P^{n-1}_B$.  Heuristically, we would like to say
that $R_{\Det(A)}$ is the $B$-algebra of global functions of $T_{A(y)}$ and
then we would have an $R_{\Det(A)}$-module $\Gamma(\O_{T_{A(y)}}(1))$.  However, we have already defined $R_{\Det(A)}$
as the $B$-algebra of global functions of $T_{A(x)}$ which is not always isomorphic to $\Gamma(\O_{T_{A(y)}})$.
Furthermore, $\Gamma(\O_{T_{A(y)}})$ is not always a free rank $n$ $B$-algebra, even when $\Det(A)$ is a non-zero-divisor.
We can, however, use this geometric construction of modules from sufficiently general
$2\times n\times n$ tensors, including from the
 the universal tensor of these dimensions.  
As in Section~\ref{SS:algZ}, we work over the ring $\Lambda=\Z[u_{ijk}]$ and 
with the universal tensor $\mathcal{C}$ in $\Lambda^2 \tensor_\Lambda \Lambda^n \tensor_\Lambda \Lambda^n$
with $i,j,k$ entry $u_{i,j,k}$. 
We have a binary $n$-ic form $c=\Det(\mathcal{C}(x))$ with coefficients in $\Lambda$.  
\begin{theorem}\label{T:samealg}
 The $\Lambda$-algebra $\Gamma(\O_{T_{\mathcal{C}(y)}})$ is isomorphic to $R_c$.
\end{theorem}

\begin{proof}
Recall that $R_c=\Gamma(\O_{T_{\mathcal{C}(x)}})$ by definition.
So we need to show that we have an isomorphism of $\Lambda$-algebras 
$\Gamma(\O_{T_{\mathcal{C}(y)}})\isom \Gamma(\O_{T_{\mathcal{C}(x)}})$.  We might expect that this would follow because we had an isomorphism
of $\Lambda$-schemes $T_{\mathcal{C}(y)}\isom T_{\mathcal{C}(x)}$, but that is not the case. However, restricted to a large open subscheme of $\Spec \Lambda$, we do
get such an isomorphism.  Let $S'$ be the open subscheme of $S=\Spec \Lambda$ that is the complement of the closed subscheme 
$Z$ (defined below).  Let $T_{\mathcal{C}(y)}'=T_{\mathcal{C}(y)} \tensor_S S',$ which is an open subscheme of $T_{\mathcal{C}(y)}$.
Similarly, let $T_{\mathcal{C}(x)}'=T_{\mathcal{C}(x)} \tensor_S S',$ which is an open subscheme of $T_{\mathcal{C}(x)}$.
We will show in Lemma~\ref{L:schemeisom} that we have an isomorphism of $S'$-schemes $T_{\mathcal{C}(y)}'\isom T_{\mathcal{C}(x)}'$.

The subscheme $Z$ of $S$ will correspond to tensors that are very degenerate.  Thus we can think of 
the tensor $\mathcal{C}\tensor_S S'$ over $S'$ as the universal ``not too degenerate'' tensor.
The  $(n-1)$-minors of $\mathcal{C}(x)$ form a matrix $W(\{x_1^{n-1},x_1^{n-2}x_2,\dots,x_2^{n-1}\},\cdot,\cdot)$.
The $i,j,k$ entry of $W\in \Lambda^n \tesnor \Lambda^n \tensor \Lambda^n$ is $(-1)^{j+k}$ times the $x_1^{n-i}x_2^{i-1}$ coefficient of the determinant of the
submatrix of $\mathcal{C}(x)$ obtained by deleting the $j$th row and $k$th column.  
We can form $\Det(W(y))$, a degree $n$ polynomial in the $y_i$, and form the ideal $\mathfrak{w}$ of $\Lambda$ of its coefficients, with
$d_i$ the coefficient of $y_i^n$.  
If we do the analogous construction, starting with $\mathcal{C}(y)$, we see that the next-to-maximal minors of $\mathcal{C}(y)$ are the entries themselves.
Then we can form $\Det(\mathcal{C}(x))=c$, and form the ideal $(c_0,\dots,c_n)$ of its coefficients.
Let $Z$ be the subscheme of $S=\Spec \Lambda$ defined by the ideal $(c_0,\dots,c_n)\mathfrak{w}$.

We will prove the theorem with the following lemmas.

\begin{lemma}
 The codimension of $Z$ in $S$ is at least 2.
\end{lemma}
\begin{proof}
Suppose for the sake of contradiction that $Z$ is codimension 1.  Then either
the subscheme of $S$ defined by $(c_0,\dots,c_n)$ or the subscheme defined by $\mathfrak{w}$
must be codimension 1 and thus given by a principal ideal.  However, we note that
$c_0$ and $c_n$ are expressions is disjoint sets of the $u_{ijk}$.  In fact, $c_0$ 
only involves the $u_{1jk}$ and $c_n$ the $u_{2jk}$.  Thus $c_0$ and $c_n$ have no nontrivial divisor in the UFD $\Lambda$,
and the subscheme cut out by $(c_0,\dots,c_n)$ cannot be codimension 1.  Similarly,
$d_j$ does not involve any $u_{** j}$. Thus a common divisor of $d_1,\dots, d_n$ must be trivial, and
so the subscheme cut out by $(d_1,\dots,d_n)$ cannot be codimension 1.
We conclude the subscheme cut out by $\mathfrak{w}$, which is contained in the subscheme cut out by $(d_1,\dots,d_n)$, cannot be codimension 1.
\end{proof}

\begin{lemma}\label{L:usecodim}
 We have that the restriction map $\Gamma(\O_S) \ra \Gamma(\O_{S'})$ is an isomorphism and thus
$\Gamma(\O_{S'})$ is naturally isomorphic to $\Lambda$.
\end{lemma}
\begin{proof}
This follows because $Z$ is codimension at least 2 (e.g. by \cite[Chapter 4, Theorem 1.14]{Liu}).
\end{proof}

From this lemma we conclude that $\Gamma(\O_{T_{\mathcal{C}(x)}'})$ and $\Gamma(\O_{T_{\mathcal{C}(y)}'})$ are $\Lambda$-algebras.

\begin{lemma}\label{L:isom1}
 The restriction map $\Gamma(\O_{T_{\mathcal{C}(x)}})\ra\Gamma(\O_{T_{\mathcal{C}(x)}'})$ is an isomorphism of $\Lambda$-algebras.
\end{lemma}
\begin{proof}
Let $\pi: T_{\mathcal{C}(x)} \ra S$.
We will see in Theorem~\ref{T:modstruc} that the $\OS$-module $\pi_*(\O_{T_{\mathcal{C}(x)}})$ is locally free on $S$. 
We will show that $\pi_*(\O_{T_{\mathcal{C}(x)}})$ is isomorphic to the pushforward 
of $\pi_*(\O_{T_{\mathcal{C}(x)}'})$ to $S$, and then taking global sections will prove the lemma.  
We cover $S$ with opens $\mathcal{U}_i$ that trivialize $\pi_*(\O_{T_{\mathcal{C}(x)}})$.  
 Since $S$ is irreducible,
$\mathcal{U}_i$ is the same dimension as $S$.  
On each $\mathcal{U}_i$,
we then have that $\mathcal{U}_i\cap Z$ is at least codimension 2 in $\mathcal{U}_i$. 
Thus, on $\mathcal{U}_i$, the sheaf $\pi_*(\O_{T_{\mathcal{C}(x)}})$ is isomorphic to $\O_S^{\oplus n}$,
for which the restriction map to $S'$ is an isomorphism by Lemma~\ref{L:usecodim}.
This means that restricted to $\mathcal{U}_i$, we have that $\pi_*(\O_{T_{\mathcal{C}(x)}})$ is isomorphic to the pushforward 
of $\pi_*(\O_{T_{\mathcal{C}(x)}'})$ to $S$.
\end{proof}

\begin{lemma}
 The restriction map $\Gamma(\O_{T_{\mathcal{C}(y)}})\ra\Gamma(\O_{T_{\mathcal{C}(y)}'})$ is an isomorphism of $\Lambda$-algebras.
\end{lemma}
\begin{proof}
The sheaf $\pi_*(\O_{T_{\mathcal{C}(y)}})$ is locally free on $S$ by Theorem~\ref{T:modstruc}.
We then use the same argument as in Lemma~\ref{L:isom1}.
\end{proof}
 
\begin{lemma}\label{L:schemeisom}
 We have an isomorphism of $S'$-schemes $T_{\mathcal{C}(y)}'\isom T_{\mathcal{C}(x)}'$.
\end{lemma}
\begin{proof}
First we will give the main idea of the proof.
The idea is that we can define a correspondence between points of $T_{\mathcal{C}(y)}'$ and $T_{\mathcal{C}(x)}'$
by $\mathcal{C}(x,y)=0$.  For points in $T_{\mathcal{C}(x)}'$, we have $\Det(\mathcal{C}(x))=0$, and thus there should be some
values of $y_j$ such that $\mathcal{C}(x,y)=0$.  For these $y_j$, we have that $\mathcal{C}(y)$ sends a non-trivial vector $x$ to 0,
and thus is rank 1.  Conversely, for points in $T_{\mathcal{C}(y)}'$, we have that $\mathcal{C}(y)$ is rank 1, and thus should
send a non-trivial vector $x$ to zero, and $\mathcal{C}(x,y)=0$ implies that $\mathcal{C}(x)$ has determinant zero.
The first difficultly in making this idea rigorous is that the correspondence is only a bijection
when $\mathcal{C}$ is sufficiently non-degenerate, which is why we have had to restrict to the base $S'$.  
Over $S'$, we could prove that the tensors are sufficiently non-degenerate to give a
 bijection of field valued points of $T_{\mathcal{C}(y)}'$ and $T_{\mathcal{C}(x)}'$.
Below, we must work a bit harder to prove an isomorphism of schemes.

First we give a map $T_{\mathcal{C}(y)}'\ra T_{\mathcal{C}(x)}'$.  We will give maps from open sets of $T_{\mathcal{C}(y)}'$ to
$\P^1_{S'}$.  We will then show that the open sets cover $T_{\mathcal{C}(y)}'$.  
We will then show that these maps agree on overlaps, and finally we will show that the image lands in $T_{\mathcal{C}(x)}'$.
Given an $1\leq i\leq n$, we map $T_{\mathcal{C}(y)}'$ to $\P^1_{S'}$ via
$x_1=-\sum_{j} u_{2jk}y_j$ and $x_2=\sum_{j} u_{1jk}y_j$, which we can do on the open set
$E_k \subset T_{\mathcal{C}(y)}'$  defined as the complement of 
the ideal $(\sum_{j} u_{1jk}y_j, \sum_{j} u_{2jk}y_j )$.  
Note that $\sum_{j} u_{ijk}y_j$ is the $i,k$ entry of $\mathcal{C}(y)$,
or of the next-to-maximal minor of $\mathcal{C}(y)$.
  Suppose there was a point of $T_{\mathcal{C}(y)}'$ not
in any $E_k$.  If we write $y$ for the vector of the $y_j$'s, then at this point we have
$\mathcal{C}(y)=0$, i.e. $\mathcal{C}_1(y)=0$ and $\mathcal{C}_2(y)=0$, and thus for formal $x_i$, we have 
$\mathcal{C}(x,y)=0$, and thus 
$\Det(\mathcal{C}(x))=0$ at this point, which contradicts
our choice of $S'$ to be in the complement of $(c_0,\dots,c_n)$.

The fact that these maps agree on the intersection of $E_k$ and $E_\ell$ is exactly given by the fact 
that on $T_{\mathcal{C}(y)}'$ the $2\times 2 $ minor of $\mathcal{C}(y)$ including
rows $k$ and $\ell$ is 0.  To see that the image of our map lands in $T_{\mathcal{C}(x)}'$, we check on open $P_i$
of $T_{\mathcal{C}(y)}'$, where $y_i$ is non-zero.  
We have that $\mathcal{C}(x)$ on $E_\ell$ has $j,k$ entry $-u_{1,j,k}\sum_a a_{2 a \ell} y_a +
u_{2,j,k}\sum_a a_{1 a \ell} y_a  $, and thus $\mathcal{C}(x,y)$ has $k$th entry
\begin{align*}
 \sum_{j} y_j \left(-u_{1,j,k}\sum_a u_{2 a \ell} y_a +
u_{2,j,k}\sum_a u_{1 a \ell} y_a \right)&=
\sum_{j} -u_{1,j,k} y_j \sum_a u_{2 a \ell} y_a +
\sum_j u_{2,j,k}y_j \sum_a u_{1 a \ell} y_a \\ &=-\mathcal{C}(y)_{1,k} \mathcal{C}(y)_{2,\ell}+\mathcal{C}(y)_{2,k} \mathcal{C}(y)_{1,\ell},
\end{align*}
which is zero by the definition of $T_{\mathcal{C}(y)}'$.
On $P_i$ we form the column vector $y/y_i$ of regular functions,
with $j$th entry $y_j/y_i$, and we see that $\mathcal{C}(x,y/y_i)=0$.  Thus we can write
the $i$th row of $\mathcal{C}(x)$ as a linear combination of the other rows, and conclude that
$\Det(\mathcal{C}(x))=0$.

Next, we will give a map $T_{\mathcal{C}(x)}'\ra T_{\mathcal{C}(y)}'$, which should be seen in analogy to the map $T_{\mathcal{C}(y)}'\ra T_{\mathcal{C}(x)}'$.
We will give maps from open sets of $T_{\mathcal{C}(x)}'$ to
$\P^{n-1}_{S'}$, and show that the open sets cover $T_{\mathcal{C}(y)}'$.  
We will then show that these maps agree on overlaps, and finally we will show that the image lands in $T_{\mathcal{C}(x)}'$.
Given an $1\leq i\leq n$, we  map $T_{\mathcal{C}(x)}'$ to $\P^{n-1}_{S'}$ by
letting $y_j$ equal the $j,k$ minor of $\mathcal{C}(x)$, that is $y_j$ equals
$(-1)^{j+k}$ times the determinant of the submatrix of $\mathcal{C}(x)$ obtained by deleting the $j$th row and $k$th column.
We have defined the $y_j$'s to be a column of minors of $\mathcal{C}(x)$.
This is a well-defined map to $\P^{n-1}$ on the open set
$F_k \subset T_{\mathcal{C}(x)}'$ defined as the complement of 
the ideal of $(n-1)$-minors of $\mathcal{C}(x)$ for the $k$th column. 
  Suppose there was a point of $T_{\mathcal{C}(x)}'$ not
in any $F_k$, then at this point we have
all minors of $\mathcal{C}(x)$ are 0.  This means that $W(\{x_1^{n-1},x_1^{n-2}x_2,\dots,x_2^{n-1}\},\cdot,\cdot)=0$.  Thus
for formal $y$, we have $W(\cdot,y,\cdot)$ has a non-trivial kernel and thus $\Det(W(y))=0$, which contradicts
our choice of $S'$ to be in the complement of $\mathfrak{w}$.

The fact that these maps agree on the intersection of $F_k$ and $F_\ell$ is exactly given by the fact 
that the
$2 \times 2$ minors of the classical adjoint matrix are divisible by the determinant of the original matrix.
  To see that the image of our map lands in $T_{\mathcal{C}(y)}'$, we check on opens $P_i$
of $T_{\mathcal{C}(x)}'$, where $x_i$ is non-zero.  
On $P_i$ we form the column vector $x/x_i$ of regular functions,
with $j$th entry $x_j/x_i$.  
Computing $\mathcal{C}(x,y)_{\ell}$ with the $y_i$'s we have defined on $F_k$ is the same as computing the determinant of $\mathcal{C}(x)$ with the
$k$th column replaced by the $\ell$th column.  Whether or not $k=\ell$, since $\Det(\mathcal{C}(x))=0$, we obtain $\mathcal{C}(x,y)_{\ell}=0$.
Thus, $\mathcal{C}(x/x_i,y)=0$ and we can write
the $i$th aisle of $\mathcal{C}(y)$ as multiple of the other aisle.  We conclude that
$\mathcal{C}(y)$ has all $2$ by $2$ minors $0$.

Now we need to check that the maps we have just given are inverses one one another.
We first check on the inverse image of $E_k$ in $F_\ell$.  Here we start with $x_i$, we define new $y_j$,
and then from the $y_j$ we define new $x_i'$.  We will compute
$-x_1'x_2+x_2'x_1$. 
Since we have 
$x_1'=-\sum_{j} u_{2jk}y_j$ and $x_2'=\sum_{j} u_{1jk}y_j$, we have that
$$
-x_1'x_2+x_2'x_1=\sum_{j} (u_{1jk}x_1+u_{2jk}x_2)y_j.
$$
We note that $u_{1jk}x_1+u_{2jk}x_2=\mathcal{C}(x)_{j,k}$, and that
$y_j$ is defined to be the $j,\ell$ minor of $\mathcal{C}(x)$.
Thus $-x_1'x_2+x_2'x_1$ is the determinant of the matrix obtained from $\mathcal{C}(x)$
by replacing the $\ell$th column by the $k$th column, and is zero in any case on $T_{\mathcal{C}(x)}$.
This shows that our maps compose to the identity on the inverse image of $E_k$ in $F_\ell$ for all $k$ and $\ell$,
and thus on $T_{\mathcal{C}(x)}$.

We now check on the inverse image of $F_k$ in $E_\ell$.
Here we start with $y_j$, we define new $x_i$,
and then from the $x_i$ we define new $y_j'$.  
At first, we will use formal $y_j$ (i.e. not assuming the relation in $T_{\mathcal{C}(y)}$).  Then
we will compute $y_j'y_m-y_m'y_j$ is in the ideal of relations $\mathcal{M}(\mathcal{C}(y))$ that cut out $T_{\mathcal{C}(y)}$.
We can form an $n\times n$ matrix $M$ with $a,b$ entry $(-1)^a\mathcal{C}(x)_{a,b}y_a$ if $a=j,m$ and $\mathcal{C}(x)_{a,b}$ otherwise.
Note that $(-1)^m y_j' y_m$ is the $j,k$ minor of $M$ and 
$(-1)^j y_j' y_m$ is the $m,k$ minor of $M$.
For any matrix $N$ the difference of $(-1)^m$ times the $j,k$ minor of $N$ and $(-1)^j$ times the $m,k$ minor of $N$
is in the ideal generated by maximal minors of $\bar{N}$, which is obtained from $N$
by deleting rows $j$ and $m$ and adding a row that is the $j$th row of $N$ plus $(-1)^{j+m}$ times the $m$th row of $N$.
The maximal minors of $\bar{M}$ are not changed if we add multiples of the original
(non-deleted) rows of $M$ to the new row of $\bar{M}$.  We add $(-1)^{j}y_a$ times the original $a$th row
of $M$ to the new row of $\bar{M}$ to obtain $\bar{M}'$.  The maximal minors of $\bar{M}'$ certainly lie
in the ideal generated by the elements of its ``new'' row, and we claim these elements are in $\mathcal{M}(\mathcal{C}(y))$.  
In the $b$th column, the element in the new row of $\bar{M}'$ is
is 
$$\sum_{i,c} u_{icb}x_iy_c(-1)^j=\sum_{i,c,a} (-1)^{i+j} u_{icb} u_{(3-i)a\ell} y_ay_c=\sum_{i}(-1)^{i+j} \mathcal{C}(y)_{i,b} \mathcal{C}(y)_{3-i,\ell} ,$$
which is the $b,\ell$ minor of $\mathcal{C}(y)$ and thus in $\mathcal{M}(\mathcal{C}(y))$.
This shows that our maps compose to the identity on $T_{\mathcal{C}(x)}$.
\end{proof}

Thus it follows that we have an isomorphism of $\Lambda$-algebras  $R_c=\Gamma(\O_{T_{\mathcal{C}(x)}})\isom \Gamma(\O_{T_{\mathcal{C}(y)}})$.
\end{proof}

From Theorem~\ref{T:samealg}, we have an $R_c$-module structure on $\Gamma(\O_{T_{\mathcal{C}(y)}}(1))$.  
We now see that it is related to the module $M_\mathcal{C}$ we constructed in Section~\ref{S:Raction}
from the universal tensor. 

\begin{theorem}\label{T:modag}
 We have an isomorphism of $R_c$-modules $$\Gamma(\O_{T_{\mathcal{C}(y)}}(1))\isom \Hom_\Lambda(M_\mathcal{C},\Lambda),$$ where $M_\mathcal{C}$ is 
as in the construction $\psi$ of two $R_c$ modules $M_\mathcal{C}$ and $N_\mathcal{C}$ given in Section~\ref{S:Raction}.
\end{theorem}

\begin{proof}
We will see in Theorem~\ref{T:modstruc} that $\Gamma(\O_{T_{\mathcal{C}(y)}}(1))$ is a free $\Lambda$-module with basis $y_1,\dots,y_n$.
Thus, it just suffices to check that the $\zeta_i$ acts on the $y_j$
in a way corresponding to their action on $M_\mathcal{C}$.
We know that $\zeta_i$ acts on the $y_j$ by a matrix of elements of $\Lambda$, and thus it suffices to determine
this action over the generic point of $\Spec \Lambda$, i.e. the fraction field of $\Lambda$.
  We have that $\mathcal{C}(x,y)=0$ and thus $y\mathcal{C}_1x_1
+y\mathcal{C}_2x_2=0$, where $y$ is a row vector of the $y_i$.  Thus, where $x_2$ is invertible, $\frac{x_1}{x_2}$
acts like $-\mathcal{C}_2\mathcal{C}_1^{-1}$ on the right on the row vector $y$, and where $x_1$ is invertible, $\frac{x_2}{x_1}$
acts like $-\mathcal{C}_1\mathcal{C}_2^{-1}$ on the right on the row vector $y$.  
Thus $\frac{x_1}{x_2}$ acts like $-\mathcal{C}_2\mathcal{C}_1^{-1}$ on the left on elements of $\Gamma(\O_{T_{\mathcal{C}(y)}}(1))$ written
as row vectors whose entries are the coefficients of the $y_i$ in the element.
We have that $\theta$ acts in elements of $M_\mathcal{C}$ by $(-\mathcal{C}_2\mathcal{C}_1^{-1})^t$ on the left.
Since in the correspondence between the algebraic
and geometric construction on $R_c$ we have that $\theta$ corresponds to $\frac{x_1}{x_2}$, we see that
the $\zeta_i$ act on $\Gamma(\O_{T_{\mathcal{C}(y)}}(1))$ as they act on $\Hom_\Lambda(M_\mathcal{C},\Lambda)$.
 \end{proof}

We can of course get a completely analogous geometric construction of $N_\mathcal{C}$ as \\
$\Hom_\Lambda(\Gamma(\O_{T_{\mathcal{C}(z)}}(1)),\Lambda)$.

\section{Geometric construction over an arbitrary base scheme}\label{S:ArbBase}
\begin{notation}
 Given a scheme $S$ and a locally free $\OS$-module $U$, we let $U^*$ denote the $\OS$-module $\sHom_{\OS} (U,\OS)$,
even if $U$ is also a module for another sheaf of algebras.  Let $\P(U)=\Proj \Sym^* U$.
\end{notation}

Now we replace $\Spec \Lambda$ by an arbitrary scheme $S$, and we consider $V,U,W$, locally free $\OS$-modules of ranks 2, $n$, and $n$,
respectively.  Let $p\in V\tesnor U \tesnor W$ denote  a global section of
$V\tesnor U \tesnor W$.  Let
$f=\Det(p)\in \Sym^n V \tensor \wn U \tensor \wn W$. 
In Section~\ref{S:arbbase} we constructed a balanced pair $M,N$ of modules for $f$ from $p$.
In this section, we will give a geometric construction of those modules, or rather
a geometric construction of $M^*$ and $N^*$ as we have done in the case of the universal form in Section~\ref{SS:geomZ}.
This construction of the modules $M^*$ and $N^*$ from $p\in V\tesnor U \tesnor W$ will
work for all $p$ and be functorial in $S$, i.e. will commute with base change in $S$.

The idea is to replace the 
subschemes of $\P(V)$, $\P(U)$, and $\P(W)$ cut out by the maximal minors of our tensor
(called $T_p(V),T_p(U),T_p(W)$, respectively, in the Introduction)
 with complexes of sheaves.  
We will then replace $\pi_*$ with the hypercohomology functors $H^0 R \pi_*$.
This has already been done
 in the construction of $R_f$ and the module $I_f$ in \cite[Section 3]{Bf}.  We face some additional challenges 
in this paper for $T_p(U)$ and $T_p(W)$ because the complexes involved are more complicated.   
One can also interpret this work as a construction of dg-schemes 
given by resolutions of the maximal minors
 instead
of just a construction of schemes.

\subsection{Arbitrary triple tensors}
We now give a more general construction before specifying to the situation of interest in this paper.
Let $S$ be an arbitrary scheme, and let $p\in V \tesnor U\tesnor W$, where $V,U,W$ are locally free $\OS$-modules of ranks
$r_V$, $r_U$, and $r_W$ respectively.  Let $r=r_V$ and assume $r_W\geq r$.  We can view $p$ as a map
$W^* \ra V\tesnor U$, and take $r$-minors of this map, with coefficients in $U$, to get
$\wedge^r_U p : \wedge^{r} W^* \ra \wedge^{r}V\tesnor \Sym^{r}U$ or equivalently
$\wedge^r_U p : \wedge^{r} W^* \tensor \wedge^{r}V^* \ra\Sym^{r}U$.

Let $\pi: \P(U) \ra S$.  Let $\O(k)$ be the usual sheaf of $\P(U)$.
Then since $\pi_* \O(r) = \Sym^{r}U$, by the adjointness of $\pi_*$ and $\pi^*$ we get a map
$$\wedge^r_U p : \pi^*\left(\wedge^{r} W^* \tensor \wedge^{r}V^*\right) \ra \O(r)$$
or equivalently, for any $k$, we get
$$\wedge^r_U p : \pi^*\left(\wedge^{r} W^* \tensor \wedge^{r}V^*\right)\tensor_{\OS} \O(k) \ra \O(r+k).$$
It is an abuse of notation to call all these maps $\wedge^r_U p$, but it better than the alternative. Locally on $S$, where $U$, $V$, and $W$ are free, the map
$\wedge^r_U p : \pi^*\left(\wedge^{r} W^* \tensor \wedge^{r}V^*\right) \ra \O(r)$ 
has image spanned by the $\binom{r_W}{r}$ $r$-by-$r$ minors of the matrix of the map
$W^* \ra V\tesnor U$, an $r_V$ by $r_W$ matrix with entries in $U$.  
The idea of our construction is to replace the sheaf $\O/\im (\wedge^r_U p)$ 
of regular functions of the subscheme of $\P(U)$ cut out by those $r$-by-$r$ minors
with complex that is generically a locally free resolution of the $\O/\im (\wedge^r_U p)$.

From the Eagon-Northcott complex, which resolves $R$ modulo the $\ell$ by $\ell$ minors of a generic matrix
(see \cite{CA}),
we can construct a
 complex
$\mathcal{C}(k)$ with $\mathcal{C}^{-1}(k) \ra \mathcal{C}^0(k)$ given by
$$\wedge^r_U p : \pi^*\left(\wedge^{r} W^* \tensor \wedge^{r}V^*\right)\tensor_{\OS} \O(k) \ra \O(r+k),$$
and with $\mathcal{C}^i(k)=0$ for $i>0$ and $i\leq -(r_W-r)-2$.  For $-(r_W-r)-1\leq i\leq -2$, we have
$$\mathcal{C}^i(k)= \pi^*(\wedge^r V^* \tesnor K_{-i}(r,W^*,V)) \tensor_\OS \O(i+1+k),$$
where $K_{-i}(r,W^*,V)$ is the locally free $\OS$-module built from $V$ and $W$ 
that is the $i$th term in the Eagon-Northcott complex for a map $\alpha: W^* \ra V$
and $d_i$ is canonically constructed from $p$ (and explained in the next paragraph).  Note that $K_{-i}(r,W^*,V)$ only depends on $V$ and $W$ and does not depend
on $\alpha$.

We now show how to construct the $d_i$.
From the construction of the Eagon-Northcott complex, 
there is a map 
$$
\sHom_{\O_Y}(W^*,V) \ra \sHom_{\O_Y}(\wedge^r V^* \tesnor K_{-i}(r,W^*,V),\wedge^r V^* \tesnor K_{-i+1}(r,W^*,V))
$$ that sends $\alpha \mapsto d_i$, where $d_i$ is the map in
the Eagon-Northcott complex for $\alpha$.  
We can extend that linear map to
$$
\sHom_{\O_Y}(W^*,V) \tensor U \ra \sHom_{\O_Y}(\wedge^r V^* \tesnor K_{-i}(r,W^*,V),\wedge^r V^* \tesnor K_{-i+1}(r,W^*,V)) \tesnor_{\O_Y} U
$$
to get the maps when there are coefficients.  Let $H_i$ be the $\O_Y$-module $\wedge^r V^* \tesnor K_{-i}(r,W^*,V)$.
We obtain $d_i : H_i \ra H_{i+1}\tensor U$, or equivalently
$d_i : H_i \tensor H_{i+1}^* \ra U$.  Using adjointness of $\pi_*$ and $\pi^*$,
this is equivalent to
$d_i : \pi^* (H_i \tensor H_{i+1}^*) \ra \O_U(1)$, 
which gives us
$d_i : \pi^* (H_i) \tensor\O_U(i+1+k)  \ra \pi^*(H_{i+1})\tensor\O_U(i+2+k)$.

The complex $\mathcal{C}(-r)$ of sheaves on $\P(U)$ has a commutative, homotopy-associative differential graded algebra structure from the commutative, homotopy-associative differential graded 
algebra structure on the Eagon-Northcott complex (which every resolution of a cyclic module has \cite[Proposition 1.1]{BE}), and the complex  $\mathcal{C}(-r+1)$ is a differential graded module for
$\mathcal{C}(-r)$. 
Now we make an important calculation about the cohomology of $\mathcal{C}(-r)$ and $\mathcal{C}(-r+1)$.

\begin{theorem}\label{T:cohcompute}
Let $p\in V \tesnor U\tesnor W$, where $V,U,W$ are locally free $\OS$-modules of ranks
$r$, $r_U$,  and $r_W$ respectively.  
Assume that $r_W\geq r\geq 2$ and that we are in one of the following cases
\begin{enumerate}
 \item $r=2$ and $r_U\geq r_W$
\item $r_U=2$ and $r=r_W$.
\end{enumerate} 
If $k=-r$ or $k=-r+1$, then
$R\pi_*\mathcal{C}(k)$ has no cohomology in any degree except 0.
\end{theorem}

\begin{proof}

Let $j\ne 0$, and we will compute that each term of the complex $\mathcal{C}(k)$ has trivial $R^j \pi_*$.
By the projection formula, we can ignore the term that is pulled back from $S$.  We have
$R^j \pi_*$ of the $i$th term $i\leq -1$ of  $\mathcal{C}(k)$, in the $i$th place, is
$R^{j-i} \pi_*$ of $\mathcal{C}^i(k)$ viewed as a complex in the $0$th place.
We have that $R^{j-i} \pi_* \O(i+1+k)=0$ unless either 1) $j=i$ and $i+1+k\geq 0$
or 2) $j-i=r_U-1$ and $i+1+k\leq -r_U$.  
Since $i+1\leq 0$ and $k\leq -r+1\leq -1$, we can never have $i+1+k\geq 0$.
We consider the two assumptions of the theorem in cases.
\begin{enumerate}
 \item {\bf Case I: $r=2$ and $r_U\geq r_W$.}  In this case, we have
$i+1+k\geq -(r_W-r)+k\geq  -(r_W-r)-r\geq -r_U$ and thus we can only have
$i+1+k\leq -r_U$ if $i=-(r_W-r)-1$, and $k=-r$, and $r_U=r_W$.
However, that implies that $i=-r_U+1$ and thus $j=0$.
\item {\bf Case II: $r_U=2$ and $r=r_W$.}  In this case, we only are considering $i=-1$, and thus
 $j-i=r_U-1$ implies $j=0$.
\end{enumerate}
We now need to consider $R^j \pi_*$ of the $0$th term of $\mathcal{C}(k)$ (for $j\ne 0$).
We have that $R^{j} \pi_* \O(r+k)=0$ unless 1)$j=0$ and $r+k\geq 0$ or
2)$j=r_U-1$ and $r+k\leq -r_U$.  However, we are assuming $j\ne0$ and $r+k\geq 0$, and thus this can never happen.
Thus we conclude that for $k=-r$ and $k=-r+1$, under our assumptions about $r_U,r,$ and $r_W$,
the complex $\mathcal{C}(k)$ has no cohomology in any degree except 0.
\end{proof}

\begin{corollary}
 Thus $\R_{\wedge^r_U p}=H^0 R\pi_* \mathcal{C}(-r)$ is a sheaf of algebras on $S$, and $\I_{\wedge^r_U p}=H^0 R\pi_* \mathcal{C}(-r+1)$ 
is a sheaf of $\R$-modules on $S$. 
\end{corollary}
\begin{proof}
Since $R\pi_* \mathcal{C}(-r)$ is equivalent to a single sheaf in degree 0, it has no non-trivial homotopies.  Thus
$H^0 R\pi_* \mathcal{C}(-r)$ has an $\OS$-algebra structure that is not just homotopy-associative but in fact associative.
Since $\mathcal{C}(-r+1)$ is a module for $\mathcal{C}(-r)$, we have that $\I_{\wedge^r_U p}=H^0 R\pi_* \mathcal{C}(-r+1)$ 
is an $H^0 R\pi_* \mathcal{C}(-r)$-module.
\end{proof}
 We can also view the construction of $\R_{\wedge^r_U p}$ 
as taking the pushforward of the regular functions on the dg-scheme given by our resolution
of $\O/\im (\wedge^r_U p)$,
instead of on the scheme cut out by $\wedge^r_U p$.

When $p$ is the universal tensor (of any size), then the Eagon-Northcott complex, and thus $\mathcal{C}(k)$,
is exact at every spot except the 0th.  Thus, when $p$ is the universal tensor, $\mathcal{C}(k)$ is quasi-isomorphic
to $\O(k)/\im(\wedge^r_U p)$. The sheaf $\O(k)/\im(\wedge^r_U p)$ is supported on the scheme defined by the $r$-by-$r$ minors in $\im(\wedge^r_U p)$,
and is isomorphic on that scheme to the pullback of $\O(k)$ from $\P(U)$.
Thus, when $\mathcal{C}$ is the universal tensor in $\Lambda^{2}\tensor \Lambda^{n}\tensor \Lambda^{n}$, we have that $\R_{\wedge^r_U \mathcal{C}}$ is the sheaf of rings given by the global sections
$\Gamma(\O_{T_{\mathcal{C}(y)}})$ (as defined in Section~\ref{SS:geomZ}), and $\I_{\wedge^r_U \mathcal{C}}$ is the $\R_{\wedge^r_U \mathcal{C}}$-module given by the global sections
$\Gamma(\O_{T_{\mathcal{C}(y)}}(1))$.

Theorem~\ref{T:cohcompute} also allows us to see that the constructions of $\R_{\wedge^r_U p}$ and $\I_{\wedge^r_U p}$ commute with base change on $S$

\begin{corollary}\label{C:BC}
Let $p\in  V\tesnor U\tesnor W$, where $U,V,W$ are locally free $\OS$-modules of ranks
$r_U$, $r$, and $r_W$ respectively.  
Assume that $r_W\geq r\geq 2$ and that we are in one of the following cases
\begin{enumerate}
 \item $r=2$ and $r_U\geq r_W$
\item $r_U=2$ and $r=r_W$.
\end{enumerate} 
Then the constructions of $\R_{\wedge^r_U p}$ and $\I_{\wedge^r_U p}$ commute with base change.
More precisely, let
$\phi : S' \ra S$ be a map of schemes.
Let $p'\in \phi^*U \tesnor \phi^*V\tesnor \phi^*W$ be the pullback of $p$.  Then the natural
map from cohomology
$$\R_{\wedge^r_U p} \tensor_{\O_S} \O_{S'} \ra \R_{\wedge^r_{\phi^*U} p'}$$ is an isomorphism of $\O_{S'}$-algebras.
Also, the 
natural
map from cohomology
$$\I_{\wedge^r_U p} \tensor_{\O_S} \O_{S'} \ra \I_{\wedge^r_{\phi^*U} p'}$$ is an isomorphism of $\R_{\wedge^r_{\phi^*U} p'}$-modules
(where the $\R_{\wedge^r_{\phi^*U} p'}$-module structure on $\I_{\wedge^r_U p} \tensor_{\O_S} \O_{S'}$ comes from
the $(\R_{\wedge^r_U p} \tensor_{\O_S} \O_{S'})$-module structure.
 \end{corollary}
\begin{proof}
The key to this proof is to compute all cohomology of the pushforward 
of the complex $\mathcal{C}(k)$ for $k=-r$ and $k=-r+1$.
We already know from Theorem~\ref{T:cohcompute} that there is
only cohomology in degree 0.
Theorem~\ref{T:modstruc} will tell us that
$H^0 R \pi_* \mathcal{C}(k)$ is locally free for $k=-r$ and $k=-r+1$.
 Thus since all $H^{i}R\pi_*(\mathcal{C}(k))$ are flat,
by \cite[Corollaire 6.9.9]{EGA3}, we have that cohomology and base change commute.
Note that the base change morphisms respect the algebra and module structures on $\R_{\wedge^r_U p}$ and $\I_{\wedge^r_U p}$,
and thus since they are isomorphisms, they are algebra and module isomorphisms.
\end{proof}

\subsection{$\OS$-module structure of $\R_{\wedge^r_U p}$ and $\I_{\wedge^r_U p}$}
Now we consider a base scheme $S$, and $V,U,W$ locally free $\OS$-modules of ranks $2$, $n$, and $n$, respectively.
In this case, we construct the $\OS$-algebra and module pairs $\R_{\wedge^n_V p}$ and $\I_{\wedge^n_V p}$,
$\R_{\wedge^2_U p}$ and $\I_{\wedge^2_U p}$, and $\R_{\wedge^2_W p}$ and $\I_{\wedge^2_W p}$.
We will now find the $\OS$-module structure of all of these constructions.  
This has already been done for the $\wedge^n_V p$ construction in \cite[Section 3.1]{Bf}, and so we consider here
the $\wedge^2_U p$ (as the $\wedge^2_W p$ constructions will follow identically).

\begin{theorem}\label{T:modstruc}
We have an exact sequence of $\OS$-modules
$$
0 \ra \OS \ra \R_{\wedge^2_U p} \ra (\Sym^{n-2} V)^*\tesnor \wt V^* \tensor \wedge^n W^*\tensor \wedge^n U^* \ra 0,
$$
and an $\OS$-module isomorphism $\I_{\wedge^2_U p}\isom U$.
\end{theorem}
\begin{proof}
From Theorem~\ref{T:cohcompute}, we know that
for $k=-r$ and $k=-r+1$
$\mathcal{C}(k)$ has trivial $H^j R\pi_*$ for all $j\ne 0$, and all components $C(k)^{i}_i$ of the complex  
(the $i$th term of $C(k)$ sitting as a complex in the $i$th place) have $H^0 R\pi_*(C(k)^{i}_i)=0$ except for possibly the
two extremal terms $i=0$ and $i=-n+1$.  Thus, by the long exact sequence of cohomology, we have the exact sequences
$$
0 \ra H^0 R\pi_* (\O) \ra \R_{\wedge^2_U p} \ra H^{n-1} R\pi_* (\pi^*(\wedge^2 V^* \tesnor K_{-n+1}(2,W^*,V)) \tensor_\OS \O(-n))\ra 0
$$
and
$$
0 \ra H^0 R\pi_* (\O) \ra \R_{\wedge^2_U p} \ra H^{n-1} R\pi_* (\pi^*(\wedge^2 V^* \tesnor K_{-n+1}(2,W^*,V)) \tensor_\OS \O(-n+1))\ra 0.
$$
We see that 
\begin{align*}
&H^{n-1} R\pi_* (\pi^*(\wedge^2 V^* \tesnor K_{-n+1}(2,W^*,V)) \tensor_\OS \O(-n))\\
=&
\wedge^2 V^* \tesnor K_{-n+1}(2,W^*,V) \tensor_\OS H^{n-1} R\pi_* ( \O(-n))\\
=&
\wedge^2 V^* \tesnor K_{-n+1}(2,W^*,V)  \tensor_\OS \wn U^*.
\end{align*}
Since $K_{-n+1}(2,W^*,V)=(\Sym^{n-2}V)^*\tensor \wedge^n W^*$, we obtain the exact sequence desired.
Also, note that $H^{n-1} R\pi_* (\O(-n+1))=0$.
\end{proof}

We can see that the three $\OS$-algebras constructed from a tensor $p\in V \tensor U\tensor W$ are isomorphic.
In the case that $V,U$, and $W$ are free, then $p$ is a pull-back from the universal tensor, in which case we know the algebras are isomorphic from Theorem~\ref{T:samealg}.  If one checks that the algebra isomorphism given by 
Theorem~\ref{T:samealg} is canonical, that it doesn't depend on the choice of bases of $V$, $U$, and $W$, 
then that would show that the three $\OS$-algebras constructed from a tensor $p\in V \tensor U\tensor W$ are all isomorphic
because locally, $V$, $U$ and $W$ are free, and if the isomorphisms between algebras do not depend on the choice of bases, they
will agree on overlaps.  In fact, the constructions made in Lemma~\ref{L:schemeisom} to give the 
isomorphism of $S'$-schemes $T_{\mathcal{C}(y)}'\isom T_{\mathcal{C}(x)}'$ are all given by minors of matrices and in fact
 are canonical.

Finally, we can see that the construction of the $\R_{\wedge^2_U p}$-module
$\I_{\wedge^2_U p}^*$ from this section  agrees with the construction of a module structure on $U^*$  given in the proof of Theorem~\ref{T:ARB}.  By Theorem~\ref{T:modstruc} $\I_{\wedge^2_U p}^*$ is also a $\R_{\wedge^2_U p}$-module structure on $U^*$.  Thus, we just need to check that the $\R_{\wedge^2_U p}$ actions agree, which we can check locally.  By Corollary~\ref{C:BC}, the construction of $\I_{\wedge^2_U p}^*$ commutes with base change,
and the construction of a module structure on $U^*$  from the proof of Theorem~\ref{T:ARB} commutes with base change by construction.  Thus it suffices to check that the module structures agree in the case of the universal tensor, which was done in Theorem~\ref{T:modag}.

\section*{Acknowledgements}
The author would like to thank Manjul Bhargava for asking the questions that inspired this research, guidance along the way, and helpful feedback both on the ideas
and the exposition in this paper.  
  This work was done as part of the author's Ph.D. thesis at Princeton University, and during the work she was supported by an NSF Graduate Fellowship, an NDSEG Fellowship, an AAUW Dissertation Fellowship, and a Josephine De K\'{a}rm\'{a}n Fellowship.  This paper was prepared for submission while the author was supported by an American Institute of Mathematics Five-Year Fellowship and National Science Foundation grant DMS-1001083.

\end{document}